\documentclass{amsart}

\usepackage{amssymb}
\usepackage{amsmath}
\usepackage{amsfonts}
\usepackage{geometry}
\usepackage{bbm}
\usepackage{hyperref}
\usepackage{tikz}
\usetikzlibrary{matrix,arrows,decorations.pathmorphing}
\usepackage{graphicx}
\usepackage{subfigure}
\usepackage{subfloat}
\usepackage[export]{adjustbox}
\usepackage{xcolor}

\newcounter{cprop}[section]

\newtheorem{theorem}[cprop]{Theorem}

\theoremstyle{plain}

\newtheorem{corollary}[cprop]{Corollary}

\newtheorem{lemma}[cprop]{Lemma}
\newtheorem{proposition}[cprop]{Proposition}

\numberwithin{equation}{section}

\theoremstyle{definition}
\newtheorem{definition}[cprop]{Definition}
\newtheorem{example}[cprop]{Example}

\newtheorem{assumption}[cprop]{Assumption}

\theoremstyle{remark}
\newtheorem{remark}[cprop]{Remark}

\newcommand{\E}{\mathbb{E}}

\newcommand{\diff}[1]{\,\mathrm{d}#1}
\newcommand{\R}{\mathbb{R}}
\newcommand{\N}{\mathbb{N}}

\newcommand{\ee}{\mathrm{e}}

\newcommand{\vertiii}[1]{{\left\vert\kern-0.25ex\left\vert\kern-0.25ex\left\vert #1 
    \right\vert\kern-0.25ex\right\vert\kern-0.25ex\right\vert}}

\begin{document}
\title{Semi-implicit Taylor schemes for stiff rough differential equations}

\author{Sebastian Riedel}
\address{Sebastian Riedel \\
Institut f\"ur Mathematik, Technische Universit\"at Berlin, Germany and Weierstra{\ss}-Institut, Berlin, Germany}
\email{riedel@math.tu-berlin.de}

\author{Yue Wu}
\address{Yue Wu \\
Mathematical Institute, University of Oxford, Oxford, UK and Alan Turning Institute, London, UK}
\email{yue.wu@maths.ox.ac.uk}

\renewcommand{\thefootnote}{\fnsymbol{footnote}} 
\renewcommand{\thefootnote}{\arabic{footnote}} 

%

 \begin{abstract}
   We study a class of semi-implicit Taylor-type numerical methods that are easy to implement and designed to solve multidimensional stochastic differential equations driven by a general rough noise, e.g. a fractional Brownian motion. In the multiplicative noise case, the equation is understood as a rough differential equation in the sense of T.~Lyons. We focus on equations for which the drift coefficient may be unbounded and satisfies a one-sided Lipschitz condition only. We prove well-posedness of the methods, provide a full analysis, and deduce their convergence rate. Numerical experiments show that our schemes are particularly useful in the case of stiff rough stochastic differential equations driven by a fractional Brownian motion.
 \end{abstract}

\maketitle

\section{Introduction}

Stiff differential equations, frequently encountered in practice, pose a challenging problem for the numerical simulation, both for deterministic and for stochastic systems. For stiff ordinary differential equations, it is well known that implicit methods typically perform better than explicit ones \cite{HW96} and they are usually also the method of choice for stiff stochastic differential equations \cite{KP92} (although one should not blindly follow this rule, cf. \cite{LAE08}). A typical stiff equation possesses one or more coefficients that are unbounded, e.g. linear, with linear growth or even satisfy one-sided growth conditions. In this paper, we will concentrate on drift coefficients which satisfy a one-sided Lipschitz condition.

Stochastic differential equations are usually modeled with white noise, justified by its universality property. However, many recent works indicate that processes with fractional noise might be more appropriate for certain models. For instance, fractional noise has been successfully used in mathematical finance \cite{Gua06,CS17,GJR18,BFG16,EER19} and in models for electricity markets \cite{Ben17}. From a mathematical point of view, the analysis of these equations is usually more challenging due to the memory in the model and, consequently, the lack of the Markov property. In the multiplicative noise case, it is even not clear how to properly define the equations since It\=o's theory of stochastic integration does not apply for non-semimartingales such as the fractional Brownian motion. However, this issue can be overcome using Lyons' theory of rough paths \cite{Lyo98} which provides a deterministic theory powerful enough to deal with such equations.

Our aim in the present paper is to define and study numerical schemes that are suitable to solve stiff stochastic differential equations driven by a general rough noise. Inspired by \cite{KPS91,KP92,Hig00}, we will concentrate on methods where the implicit parameter appears in the drift component only, usually called \emph{semi-implicit methods}. They are conceptually easier than fully-implicit methods and are known to perform well for stiff stochastic differential equations with additive noise and when the noise parameter is not too large. The most studied numerical schemes in the context of rough differential equations are Taylor-type schemes \cite{Dav07,FV10,DNT12,BFRS16} (see, however, the recent works \cite{HHW18,RR20} for an approach to Runge-Kutta methods), and we will study such methods in the present work, too. On the technical level, the biggest challenge is the presence of an unbounded drift in the equation which satisfies a one-sided Lipschitz condition only. In the case of additive noise, we can use a direct calculation to show that the numerical scheme remains bounded under this condition, cf. Section \ref{sec:add}. Interestingly, a one-sided growth assumption (which follows by the one-sided Lipschitz property) is not sufficient to guarantee this, cf. \cite[page 43]{CHJ13} for a counterexample. In the multiplicative noise case, things are getting much more complicated. In the work \cite{RS17}, the authors can show that a rough differential equation with an unbounded drift has a global solution provided that the drift satisfies a further growth assumption in the normal directions, cf. \eqref{eqn:normal_growth}. Our strategy in the present article is to impose that the continuous equation has a global solution and to derive under this assumption the boundedness and convergence of the numerical scheme. The advantage of this approach is that our results can be applied to any continuous equation, regardless of the precise assumptions on the vector fields, as long as the global solution exists. 

The paper is structured as follows. In Section \ref{sec:add}, we study the case of an ordinary differential equation perturbed by additive noise. The implicit Euler scheme is defined in \eqref{eqn:RN_y}. We show the convergence of the numerical solution to the true solution with a precise rate in Theorem \ref{thm:RDE_y}. In Section \ref{sec:semiflow}, we study equations driven by multiplicative noise interpreted as rough differential equations. We define semi-implicit Euler \eqref{eqn:implicit_Euler}, Milstein \eqref{eqn:implicit_Milstein} and 3rd-order Milstein schemes \eqref{eqn:implicit_Milstein_order3} and analyze their convergence rates in Theorem \ref{thm:conv_rates}. Inspired by \cite{DNT12}, we also propose \emph{simplified} versions of the respective schemes, i.e. we replace the iterated (stochastic) integrals by a product of increments. These schemes are much easier to implement in practice. The convergence rate for these schemes are studied in Theorem \ref{thm:rates_Gaussian_rough} for a driver being a general Gaussian process. In Section~\ref{sec:numexp}, we illustrate our theoretical results through several numerical experiments for equations driven by a fractional Brownian motion. For both additive and multiplicative noise, the divergence of the forward Euler scheme with coarser step size is also discussed to illustrate the drawback of the forward schemes, while the semi-implicit schemes always return reliable simulations regardless of stepsize. This observation is somehow crucial to applications in the real world which may require less computational costs.

\section{Preliminaries and notation}\label{sec: pre}
This section introduces basic notations and useful mathematical results for both Section \ref{sec:add} and \ref{sec:semiflow}. Notations which are used exclusively in Section \ref{sec:semiflow} will be postponed to  the beginning of Section \ref{sec:semiflow}.

Let $T \in (0,\infty)$. Let $(\R^d, |\cdot|)$ be the Euclidean space equipped
with the Euclidean distance. By $\langle\cdot,\cdot\rangle$ we denote the
scalar product in $\R^d$.

First, let us recall the notion of $\alpha$-H\"older and
$p$-variation regularity from \cite[Chapter 5]{FV10}: 

\begin{definition}
  \label{def:p-varholder}
   A path $x \colon [0,T]\to \R^d$ is said to be \\ 
  \textbf{a}. \emph{$\alpha$-H\"older continuous} with $\alpha \in (0,1]$ if
  \begin{equation}
    \label{eqn:p-holder}
    \|x\|_{\alpha;[0,T]}=\sup_{[s,t]\subset
    [0,T]}\frac{|x(t)-x(s)|}{|t-s|^{\alpha}}<\infty; 
  \end{equation}
  \textbf{b}. of {\emph{finite $p$-variation}} for some $p \in [1, \infty)$ if
  \begin{equation}
    \label{eqn:p-var}
    \|x\|_{p\text{-var};[0,T]} = \sup_{(t_i)\subset [0,T]}\Big(\sum_i
    |x({t_{i+1}})-x({t_{i}})|^p\Big)^{\frac{1}{p}}<\infty, 
  \end{equation}
  where the supremum is taken over all finite partitions of the interval
  $[0,T]$. 
\end{definition}

The notation $C^{\alpha\text{-H\"ol}}([0,T],\R^d)$ is used for the set of
$\alpha$-H\"older paths $x$, which can be shown to be a Banach space with norm
$x\mapsto |x(0)|+\|x\|_{\alpha;[0,T]}$. The notation
$C^{p\text{-var}}([0,T],\R^d)$ will be used for the set of continuous $x \colon
[0,T]\to \R^d$ of finite $p$-variation. Indeed, $C^{p\text{-var}}([0,T],\R^d)$
can be shown to be a Banach space with norm $x\mapsto
|x(0)|+\|x\|_{p\text{-var};[0,T]}$. 

\begin{remark}
  Note that $C^{\alpha\text{-H\"ol}}([0,T],\R^d)$ and
  $C^{p\text{-var}}([0,T],\R^d)$ are subsets of $C([0,T],\R^d)$, the collection
  of all continuous path $x$, with norm $\|x\|_{\infty;[0,T]}:=\sup_{t\in
  [0,T]}|x(t)|$. It is not hard to deduce that 
  \begin{align}
    \label{eqn:holder-infinity}
    \|x\|_{\infty;[0,T]}\leq |x(0)|+\|x\|_{\alpha;[0,T]}T^\alpha
  \end{align}
  if $x\in C^{\alpha\text{-H\"ol}}([0,T],\R^d)$ and
  \begin{align}
    \label{eqn:p-infinity}
    \|x\|_{\infty;[0,T]}\leq |x(0)|+\|x\|_{p\text{-var};[0,T]}
  \end{align}
  if $x\in C^{p\text{-var}}([0,T],\R^d)$.
  
  Furthermore, one easily verifies that $C^{\alpha\text{-H\"ol}}([0,T],\R^d)
  \subset C^{p\text{-var}}([0,T],\R^d)$ if $\alpha \ge \frac{1}{p}$ as well as 
  $C^{p\text{-var}}([0,T],\R^d) \subset C^{p'\text{-var}}([0,T],\R^d)$ if $1
  \le p \le p' < \infty$. A more detailed discussion can be found in
  \cite[Chapter 5]{FV10}. 
\end{remark}

In the discussion of $p$-variation regularity, the concept of a {\emph{control
function}} is very useful: 

\begin{definition}
  \label{def:control}
  Let $\Delta_T:=\{ (s,t) \in [0,T] \times [0,T] \, : \,
  0\leq s\leq t\leq T\}$ denote the simplex.  A continuous 
  non-negative function $\omega \colon \Delta_T \to [0,\infty)$ with
  $\omega(t,t)=0$ for all $t\in [0,T]$ is called a \emph{control function}
  if it is \emph{superadditive}, i.e., 
  $$\omega(s,t)+\omega(t,u)\leq \omega(s,u) \ \ \ \mbox{for\ all\ } s\leq t
  \leq u\ \ \  \mbox{in\ } [0,T].$$
  If there exists a positive constant $C$ such that
  \begin{align}
    \label{eqn:control}
    |x(t)-x(s)|^p\leq C\omega(s,t) \ \ \mbox{for\ every\ } s\leq t,
  \end{align}
  we say that \emph{$\omega$ controls the $p$-variation of $x$}. 
\end{definition}

Note that \eqref{eqn:control} immediately implies that
$\|x\|_{p\text{-var};[0,T]} \leq \omega(0,T)^{\frac{1}{p}}$. On the other hand,
for $x\in C^{p\text{-var}}([0,T],\R^d)$, the function
$\omega(s,t)=\|x\|^p_{p\text{-var};[s,t]}$ is a control function which controls
the $p$-variation of $x$ (\cite[Proposition 5.8]{FV10}). 

Let us finally recall two versions of the Gronwall inequality \cite{G1919}.
The first version states the Gronwall inequality in a differential form.
For a proof we refer, for instance, to \cite[Lemma~7.3.2]{emmrich2004}.

\begin{lemma}[Differential version of Gronwall's inequality]
  \label{lem:Gronwall}
  Assume that $a \colon [0,T] \to \R$ is absolutely continuous and $g,
  \lambda \colon [0,T] \to \R$ are integrable, i.e.,
  $g, \lambda \in L^1([0,T],\R)$. If it holds
  $$ \dot{a}(t) \leq g(t)+\lambda(t)a(t), \ \ \mbox{for almost every }t\in
  [0,T].$$
  Then, it follows that
  \begin{align*}
    a(t) \leq \ee^{\Lambda(t)} a(0) + 
    \int_{0}^{t} \ee^{\Lambda(t)-\Lambda(s)} g(s)\diff{s}, \quad \text{ for all
    } t \in [0,T],
  \end{align*}
  where $\Lambda(t) := \int_{0}^t\lambda(s)\diff{s}$.
\end{lemma}

In addition, we also rely on the discrete version of the Gronwall inequality.
A proof of this can be found in \cite{clark1987}. For its formulation
we use the convention that a sum over an empty index set is equal to zero.

\begin{lemma}[Discrete version of Gronwall's inequality]
  \label{lem:Gronwall-D} 
  Let $(u_n)_{n\in \N}$ and $(b_n)_{n\in \N}$ be two nonnegative sequences
  which satisfy, for given $a\in [0,\infty)$ and $N \in \N$, that
  \begin{align*}
    u_n \leq a + \sum_{i=1}^{n-1} b_i u_i, 
    \quad n \in \{1,\dots, N \}.
  \end{align*}	
  Then, it follows that
  \begin{align*}
    u_n \leq a \exp \Big( \sum_{i=1}^{n-1} b_i \Big), 
    \quad n \in \{1,\dots, N \}.
  \end{align*}
\end{lemma}

Throughout this paper, the drift term $b \colon \R^d \to \R^d$ in the equation considered is spatially dependent only and assumed to satisfy the following assumption:

\begin{assumption}
  \label{as:drift}
    The vector field $b \colon \R^d \to \R^d$ is continuous and satisfies
    a \emph{one-sided Lipschitz condition}, i.e. there exists a 
    constant $C_b \in [0,\infty)$ such that 
    \begin{equation}
      \label{eqn:onelip}
      \left\langle b(\varsigma)-b(\zeta), \varsigma-\zeta \right\rangle
      \leq C_b|\varsigma-\zeta|^2\ \ \  \forall \varsigma,\zeta\in\R^d.
    \end{equation}
\end{assumption}

\begin{assumption}
  \label{as:drift2}
  The vector field $b \colon \R^d \to \R^d$ is \emph{locally Lipschitz
  continuous}. To be more precise, there exists a non-decreasing function
  \mbox{$\Upsilon_b \colon [0,\infty)\to [0,\infty)$} such that for every 
  $r \geq 0$, 
  \begin{equation}
    \label{as:LocalLipBounded}
    |b(\varsigma)-b(\zeta)| \leq \Upsilon_b(r)|\varsigma-\zeta|\ \ \
    \forall \varsigma,\zeta \in \overline{B_r(0)},
  \end{equation}
  where $\overline{B_r(0)}$ denotes the closed ball in $\R^d$ centered at
  $0$ with radius $r$. 
\end{assumption}

\section{Additive noise}\label{sec:add}

In this section, we consider rough differential equations of the form 
\begin{align}\label{eqn:add_noise}
  \begin{split}
    \begin{cases}
      \diff{y}(t) = b(y(t))\diff{t} + \diff x(t), \quad  t \in (0,T],\\
      y(0) = \xi,
    \end{cases}
  \end{split}
\end{align}
where $\xi \in \R^d$ is the initial condition, $b \colon \R^d \to \R^d$ is a
vector field and $x \colon [0,T] \to \R^d$ is a given path. The equation is understood as an integral equation, i.e. $y$ is the solution to \eqref{eqn:add_noise} if and only if
\begin{align}\label{eqn:add_sol}
     {y}(t) =  \xi+\int_0^t b(y(s))\diff{s} + x(t)-x(0), \quad  t \in (0,T].
\end{align}

\subsection{Global existence and uniqueness}
To solve \eqref{eqn:add_noise}, we will use the following transformation:
For $t \in [0,T]$ and $\zeta \in \R^d$ set   
\begin{equation}
  \label{eqn:transformation}
  f(t,\zeta):=b\big(\zeta+x(t)\big).
\end{equation}
Then we see that, formally, the solution $y$ to \eqref{eqn:add_noise} is given
by $y(t) = z(t) + x(t)$, where $z$ is a solution to the initial value problem
\begin{align}
  \label{eqn:ODE_transform}
  \begin{cases}
    \dot{z}(t) = f(t,z(t)), \quad t \in (0,T],\\
    z(0) = \xi.
  \end{cases}
\end{align}
Therefore, the problem of solving \eqref{eqn:add_noise} reduces to solving
\eqref{eqn:ODE_transform}. 

In the following lemma, we collect some properties of the non-homogeneous
vector field $f$.

%

\begin{lemma}
  \label{lm:fproperties} 
  Let $x \colon [0,T] \to \R^d$ be bounded. 
  Consider the mapping $f$ defined by \eqref{eqn:transformation}.
  \begin{itemize}
    \item[(i)] If Assumption~\ref{as:drift} is satisfied then
      $f$ is one-sided Lipschitz continuous in space, uniformly in
      time, i.e.,  
      \begin{equation}
        \label{as:Lipf}
        \left\langle f(t,\varsigma)-f(t,\zeta),\varsigma-\zeta\right\rangle
        \leq C_b|\varsigma-\zeta|^2\ \ \  \forall \varsigma,\zeta\in\R^d\
        \mbox{and\ } t\in [0,T]. 
      \end{equation}
    \item[(ii)] If Assumption~\ref{as:drift2} is satisfied then for every $r
      \in (0,\infty)$, $\varsigma,\zeta \in \overline{B_r(0)}$ and $t\in
      [0,T]$, it holds
      \begin{equation}
        \label{as:LocalLipf}
        |f(t,\varsigma)-f(t,\zeta)|
        \leq \Upsilon_b\big(\rho_r(t)\big) |\varsigma-\zeta|,
      \end{equation}
      where $\rho_r(t) = r + \| x \|_{\infty;[0,t]}$ is a non-decreasing
      function.  
  \end{itemize}
\end{lemma}

\begin{proof}
Straightforward.
\end{proof}


\begin{theorem}
  \label{thm:RDE}
  Let $x \colon [0,T] \to \R^d$ be continuous and assume that
  Assumptions~\ref{as:drift} and \ref{as:drift2} are satisfied. Then there
  exists  a unique, global solution $z$ to \eqref{eqn:ODE_transform} with
  \begin{align*}
    \|z\|_{\infty; [0,T]} \leq \ee^{C_b T}
    |\xi|+\int_0^T \ee^{C_b(T-s)}\big|b(x(s))\big| \diff{s} =: r_z.
  \end{align*}
  Moreover, $z$ is Lipschitz continuous on $[0,T]$ with Lipschitz constant
  bounded by  
  \begin{align*}
    \| z \|_{1;[0,T]} \le 
    r_z \Upsilon_b \big(r_z + \|x\|_{\infty;[0,T]}\big) 
    + \|b \circ x \|_{\infty;[0,T]}.
  \end{align*}
\end{theorem}

\begin{proof}
  The existence of a unique local solution to \eqref{eqn:ODE_transform} 
  follows directly form the local Lipschitz continuity of $f$. For instance, we
  refer to \cite{hale1980}. The norm estimates are also derived from the
  Gronwall inequality Lemma~\ref{lem:Gronwall} in a standard
  way. For further details we refer to \cite[Section 3]{EK17}. 
\end{proof}

Due to $y = z + x$ we also immediately
obtain the existence of a unique solution to \eqref{eqn:add_noise}. The result
summarizes some additional properties which also follow from
Theorem~\ref{thm:RDE}.

\begin{corollary}
  \label{cor:RDE}
  Let $x \colon [0,T] \to \R^d$ be continuous and assume that
  Assumptions~\ref{as:drift} and \ref{as:drift2} are satisfied.
  Then there exists a unique, global and 
  continuous solution $y$ to \eqref{eqn:add_noise}.

  In addition, let $\omega \colon \Delta_T \to [0,\infty)$ be 
  a control function for the $p$-variation of $x$, $p \in [1,\infty)$. Then the
  $p$-variation of $y$ is also controlled by a control function $\hat{\omega}$
  which is given by 
  \begin{align*}
    \hat{\omega}(s,t) = \|y\|_{p\mathrm{-var};[s,t]}^p \le 2^{p-1}
    \| z \|_{1;[0,T]}^p |t-s|^p + 2^{p-1} \omega(s,t), \quad (s,t)
    \in \Delta_T.
  \end{align*}
  If $x$ is $\frac{1}{p}$-H\"older continuous, then also $y$ is
  $\frac{1}{p}$-H\"older continuous with H\"older constant bounded by
  \begin{align*}
    \|y\|_{\frac{1}{p};[0,T]} \le \|x\|_{\frac{1}{p};[0,T]} +
    \| z \|_{1;[0,T]} T^{1 - \frac{1}{p}}. 
  \end{align*}
\end{corollary}

\begin{remark}
  \label{rmk:explosion}
  The above theorem shows that the global, one-sided Lipschitz condition
  \eqref{eqn:onelip} implies non-explosion of the solution in the additive
  noise case. At first sight, this seems not very surprising since
  \eqref{eqn:onelip} clearly implies the one-sided growth condition 
  \begin{align}
    \label{eqn:one-sided_growth}
    \langle b(\zeta), \zeta \rangle \leq C_1 + C_2|\zeta|^2 \quad \forall \zeta
    \in \R^d 
  \end{align}
  which is known to prevent explosion in the case of classical ODEs, i.e., $x
  \equiv 0$, or in the case of stochastic differential equation driven by a 
  Brownian motion. However, if the equation is driven by a general path, 
  non-explosion can \emph{not} be deduced by simply assuming 
  \eqref{eqn:one-sided_growth}. We refer to \cite[p.~43]{CHJ13} for a
  counterexample. 
\end{remark}

\subsection{Discrete approximations and error analysis}

Let us fix an equidistant partition $\mathcal{T}^h$ of $[0,T]$ of the form 
\begin{equation}
  \label{eqn:time_grid}
  \mathcal{T}^h =\{t_0=0<t_1<\dots <t_j <\dots<t_{N_h}=T\}\ \ \mbox{with\
  } t_j=jh. 
\end{equation}
Hereby, the step size $h \in (0,T]$ is determined by $h = \frac{T}{N_h}$ with
$N_h \in \N$. 

We will investigate the implicit Euler scheme given by
\begin{align}
  \label{eqn:RN_y}
  \begin{split}
    \begin{cases}
      y_{j+1} = y_{j} + h b(y_{j+1}) + x_{j+1}-x_{{j}}&
      \text{ for } j\in \{0,\ldots,N_h-1\},\\
      y_0=\xi,&      
    \end{cases}
  \end{split}
\end{align}
where, for each $i \in \{0,\ldots,N_h\}$, $y_{i}$ denotes the numerical
approximation of the exact solution $y(t_i)$ at time point $t_{i}$ and $x_i$ is
short for $x(t_i)$.  

The implementation of \eqref{eqn:RN_y} requires the solution of a nonlinear
equation in each time step. The following result ensures the existence of 
a unique $\mathbb{R}^d$-valued sequence $(y_j)_{0 \leq j \leq N_h}$ satisfying
the difference equation \eqref{eqn:RN_y}. Proposition~\ref{prop:nonlinear} is
a standard result in nonlinear analysis and often called \emph{Uniform
Monotonicity Theorem}. For a proof we refer to \cite[Chap.~6.4]{ortega2000}
and \cite[Theorem~C.2]{SH96}.

\begin{proposition}
  \label{prop:nonlinear}
  Let $G \colon \R^d \to \R^d$ be a continuous mapping such that
  there exists a constant $L_G \in (0,\infty)$ with
  \begin{align*}
    \langle G(\varsigma) - G(\zeta), \varsigma - \zeta \rangle \ge L_G
    |\varsigma - \zeta|^2, \quad \text{ for all } \varsigma, \zeta
    \in\R^d.
  \end{align*}
  Then $G$ is a homeomorphism with Lipschitz continuous inverse. In particular,
  it holds
  \begin{align*}
    |G^{-1}(\varsigma) - G^{-1}(\zeta)| \le \frac{1}{L_G} |\varsigma - \zeta|,
    \quad \text{ for all } \varsigma, \zeta \in\R^d.
  \end{align*}
\end{proposition}

An application of Proposition~\ref{prop:nonlinear} immediately gives the
well-posedness of the numerical method.

\begin{theorem}[Well-posedness]
  \label{thm:well-posedness}
  Let $(x_j)_{0 \leq j \leq N_h}$ be an $\R^d$-valued sequence.
  Let Assumption~\ref{as:drift} be satisfied with one-sided Lipschitz constant $C_b$. 
  If $C_b h < 1$ then there exists a unique $\mathbb{R}^d$-valued sequence
  $(y_j)_{0 \leq j \leq N_h}$ satisfying the difference equation \eqref{eqn:RN_y}. 
\end{theorem}

\begin{proof}
  Let $C_b h < 1$ and define $G \colon \R^d \to \R^d$ by $G(\zeta) =
  \zeta - h b(\zeta)$ for all $\zeta \in \R^d$. Then it holds
  \begin{align*}
    \langle G(\varsigma) - G(\zeta), \varsigma - \zeta \rangle
    = |\varsigma - \zeta|^2 - h   
    \langle b(\varsigma) - b(\zeta), \varsigma - \zeta \rangle
    \ge (1 - C_b h) |\varsigma - \zeta|^2.
  \end{align*}
  Due to $C_b h < 1$ we have $L_G := 1 - C_b h > 0$. Hence,
  Proposition~\ref{prop:nonlinear} is applicable. In particular,
  the sequence $(y_j)_{0 \leq j \leq N_h}$ defined by 
  \begin{align*}
    y_{j+1} := G^{-1}(y_{j} + x_{j+1}-x_{{j}})
  \end{align*}
  for every $j \in \{0,1,\ldots,N_h -1\}$ satisfies \eqref{eqn:RN_y}.
\end{proof}

We derive a bound for the numerical solution \eqref{eqn:RN_y}
uniformly with respect to the step size $h$.

\begin{proposition}
  \label{prop:RN_y}
  Under the same assumptions as in Theorem~\ref{thm:well-posedness}, the 
  solution $(y_j)_{0 \leq j\leq N_h}$ to the difference equation
  \eqref{eqn:RN_y} satisfies  
  $$\max_{n \in \{0,\ldots,N_h\}}|y_n|
  \leq \ee^{2C_b T } \big(|\xi| + \max_{n \in \{1,\ldots,N_h\}}
  |b(x_n)| T \big) + \max_{n \in \{0,\ldots,N_h\}}  |x_n|,$$ 
  provided $2 C_b h \leq 1$.
\end{proposition}

\begin{proof}
  Define $\hat{y}_{j}:=y_j-x_j$. Then, the recursion \eqref{eqn:RN_y} can be
  rewritten in terms of $\hat{y}_j$ by
  \begin{equation}
    \label{eqn:inter_yc}
    \hat{y}_{j+1} = \hat{y}_{j} + h b\big(\hat{y}_{j+1}+x_{j+1}\big),
    \quad \text{ for all } j \in \{0,\ldots,N_h-1\}. 
  \end{equation}
  To prove the boundedness, we use the following estimation:
  \begin{align}
    \label{eqn:estimation_bdaddi}
    \begin{split}
      \big \langle \hat{y}_{j+1}-\hat{y}_{j}, \hat{y}_{j+1}\big\rangle
      &= h\big\langle b\big(\hat{y}_{j+1}+x_{j+1}\big),
      \hat{y}_{j+1}\big\rangle\\
      &=h\big\langle
      b\big(\hat{y}_{j+1}+x_{j+1}\big)-b\big(x_{j+1}\big), \hat{y}_{j+1}\big
      \rangle+h \big\langle b\big(x_{j+1}\big), \hat{y}_{j+1}\big\rangle\\
      & \leq hC_b |\hat{y}_{j+1}|^2 +  h |b(x_{j+1})|
      |\hat{y}_{j+1}|, 
    \end{split}
  \end{align} 
  where the inequality follows from the one-sided Lipschitz
  condition and the Cauchy--Schwarz inequality.
  On the other hand, note that
  \begin{equation}
    \label{eqn:elementary}
    \begin{split}
      \langle u_2-u_1,u_2\rangle
      =  |u_2|^2 -  \langle u_1, u_2 \rangle
      & \ge |u_2| \big( |u_2| - |u_1| \big)
    \end{split}
  \end{equation}
  for all $u_1,u_2\in \mathbb{R}^d$.
  From \eqref{eqn:elementary} and estimate \eqref{eqn:estimation_bdaddi},
  we can conclude that 
  \begin{align*}
    |\hat{y}_{j+1}| \big( |\hat{y}_{j+1}| - |\hat{y}_{j}| \big)
    &\le \langle \hat{y}_{j+1}-\hat{y}_{j}, \hat{y}_{j+1}\big\rangle\\
    &\leq hC_b |\hat{y}_{j+1}|^2 +  h |b(x_{j+1})| |\hat{y}_{j+1}|
  \end{align*}
  After canceling one time $|\hat{y}_{j+1}|$ from both sides of the inequality
  we arrive at
  \begin{align*}
    |\hat{y}_{j+1}| - |\hat{y}_{j}| \le C_b h |\hat{y}_{j+1}| + h |b(x_{j+1})|
  \end{align*}
  for every $j \in \{0,\ldots,N_h-1\}$. 
   Summing both sides up to arbitrary $n \in \{1,\ldots,N_h\}$ yields 
  \begin{align*}
    |\hat{y}_{n}| &\le |\hat{y}_{0}| + C_b h \sum_{j=0}^{n-1} |\hat{y}_{j+1}|
    + h \sum_{j=0}^{n-1} |b\big(x_{j+1}\big)| \\
    &\le |\hat{y}_{0}| + C_b h \sum_{j=1}^{n-1} |\hat{y}_{j}|
    + \frac{1}{2} |\hat{y}_{n}| + h \sum_{j=0}^{n-1} |b\big(x_{j+1}\big)|
  \end{align*}
  due to $2 C_b h \le 1$.
  Then, an application of the discrete Gronwall inequality, i.e.,
  Lemma~\ref{lem:Gronwall-D}, leads to 
  \begin{align*}
    |\hat{y}_{n}| \leq 2 \ee^{2 C_b t_n} 
    \big(|\hat{y}_{0}| + \max_{j \in \{1,\ldots,n\}}
    |b(x_j) | t_n\big)
  \end{align*}
  for every $n \in \{0,\ldots,N_h\}$. Finally, after
  recalling the relationship between $y$ and $\hat{y}$, we arrive at 
  \begin{align*}
    \max_{n \in \{0,\ldots,N_h\}} |y_{n}|
    &\leq \max_{n \in \{0,\ldots, N_h\}} |\hat{y}_{n}| +
    \max_{n \in \{0,\ldots, N_h\}} |x_n|\\
    &\leq \ee^{2C_bT} \big(|\xi| + \max_{n \in \{1,\ldots,N_h\}} |b(x_n)|
    T\big) + \max_{n \in \{0,\ldots, N_h\}} |x_n|.
  \end{align*}
  This completes the proof.
\end{proof}

Based on similar arguments as in the proof of Proposition~\ref{prop:RN_y}
we can derive error estimates for the implicit Euler method. In particular, 
we show that the solution $(y_j)_{0 \leq j\leq N_h}$ to the difference equation
\eqref{eqn:RN_y} converges to the exact solution with an order $\frac{1}{p}$
depending on the regularity of the driving path $x$.
We first investigate the case if $x$ is of finite $p$-variation.

\begin{theorem}[$p$-variation case]
  \label{thm:RDE_y2}
  Let $x \colon [0,T] \to \R^d$ be continuous and of finite $p$-variation for
  some $p \in [1,\infty)$. Suppose that Assumption~\ref{as:drift} is fulfilled
  and that the initial value problem \eqref{eqn:add_noise} has a unique
  solution $y$ of finite $p$-variation satisfying
  \begin{align*}
    \| b \circ y \|_{p\mathrm{-var};[0,T]} < \infty.
  \end{align*}
  Then the implicit Euler scheme \eqref{eqn:RN_y} converges to the solution of
  \eqref{eqn:add_noise} with order $\frac{1}{p}$. To be more precise, it holds
  \begin{align*}
    \max_{n \in \{0,\ldots,N_h\}} \big|y(t_n)-y_n\big| 
    \le 2 T^{1 - \frac{1}{p}} \ee^{2C_b T} \|b \circ y
    \|_{p-\mathrm{var};[0,T]} h^{\frac{1}{p}}, 
  \end{align*}
  for all $h \leq \frac{1}{2 C_b}$.
\end{theorem}

\begin{proof}
  We denote by $e_j :=y(t_j)-y_j$ the difference between the solutions of
  \eqref{eqn:add_noise} and \eqref{eqn:RN_y} at each time step $j \in
  \{0,\ldots,N_h\}$. Then observe that
  \begin{align}
    \label{eq:term1}
    \begin{split}
      &\big \langle e_{j+1}-e_{j}, e_{j+1}\big\rangle
      = \Big\langle \int_{t_j}^{t_{j+1}}\big(b(y(s)) - b(y_{j+1})\big)\diff{s},
      e_{j+1}\Big\rangle \\
      &\quad = \Big\langle
      \int_{t_j}^{t_{j+1}}\big(b(y(s))-b(y(t_{j+1}))\big)\diff{s},
      e_{j+1}\Big\rangle + \Big\langle
      \int_{t_j}^{t_{j+1}}\big(b(y(t_{j+1}))-b(y_{j+1})\big)\diff{s},
      e_{j+1}\Big\rangle\\
      &\quad \le \int_{t_j}^{t_{j+1}} \big\langle
      b(y(s))-b(y(t_{j+1})), e_{j+1} \big\rangle \diff{s}
      + C_b h | e_{j + 1}|^2,
    \end{split}
  \end{align}
  by applying \eqref{eqn:onelip}. 
  Next, the mean value theorem for integrals yields the existence of some $\xi_j
  \in (t_{j},t_{j+1})$ with
  \begin{align*}
    \int_{t_j}^{t_{j+1}} \big\langle
    b(y(s))-b(y(t_{j+1})), e_{j+1} \big\rangle \diff{s}
    = h \big\langle b(y(\xi_j))-b(y(t_{j+1})), e_{j+1} \big\rangle.
  \end{align*}
  From \eqref{eqn:elementary} it follows that
  \begin{align*}
    |e_{j+1}| \big( |e_{j+1}| -  |e_{j}|\big)
    \le \big\langle e_{j+1}-e_{j}, e_{j+1}\big\rangle. 
  \end{align*}
  After inserting this into \eqref{eq:term1} an application of the
  Cauchy--Schwarz inequality shows that 
  \begin{align*}
    |e_{j+1}| \big( |e_{j+1}| -  |e_{j}|\big) 
    \le h \big|b(y(\xi_j))-b(y(t_{j+1}))\big| |e_{j+1}| 
    +  C_b h | e_{j + 1}|^2
  \end{align*}
  for every $j \in \{0,\ldots,N_h - 1\}$. Hence, after canceling $ |e_{j+1}|$
  from both sides of the inequality and summing up to $n$ we arrive at
  \begin{align*}
    |e_{n}| &\le  |e_{0}| 
    + h \sum_{j = 0}^{n-1} \big|b(y(\xi_j))-b(y(t_{j+1}))\big|  
    + C_b h \sum_{j = 0}^{n-1} | e_{j + 1}|\\
    &\le |e_{0}| 
    + h \sum_{j = 0}^{n-1} \big|b(y(\xi_j))-b(y(t_{j+1}))\big|  
    + C_b h \sum_{j = 1}^{n-1} | e_{j}| + \frac{1}{2} |e_{n}|,
  \end{align*}
  where the last step follows from $2 C_b h \le 1$. Therefore, we have shown
  that 
  \begin{align*}
    |e_{n}| &\le 2|e_{0}| + 2 h \sum_{j = 0}^{n-1}
    \big|b(y(\xi_j))-b(y(t_{j+1}))\big| 
    + 2 C_b h \sum_{j = 1}^{n-1} | e_{j}|
  \end{align*}
  for every $n \in \{1,\ldots,N_h\}$.
  Hence, the discrete Gronwall inequality, Lemma~\ref{lem:Gronwall-D}, is
  applicable and yields
  \begin{align*}
    \max_{n \in \{1,\ldots,N_h\}} |e_{n}|
    \le 2 \ee^{2 C_b T} \Big( |e_{0}| + h \sum_{j = 0}^{n-1}
    \big|b(y(\xi_j))-b(y(t_{j+1}))\big| \Big).
  \end{align*}
  Finally, an application of the H\"older inequality with $1 =
  \frac{1}{p} + \frac{1}{q}$ gives
  \begin{align*}
    h \sum_{j = 0}^{n-1}
    \big|b(y(\xi_j))-b(y(t_{j+1}))\big|
    &\le \Big( \sum_{j = 0}^{N_h-1} h^q \Big)^{\frac{1}{q}} 
    \Big( \sum_{j = 0}^{N_h-1} \big|b(y(\xi_j))-b(y(t_{j+1}))\big|^p
    \Big)^{\frac{1}{p}}\\
    &\le T^{\frac{1}{q}} h^{1 - \frac{1}{q}} 
    \| b \circ y \|_{p-\mathrm{var};[0,T]}.
  \end{align*}
  Since $e_0 = 0$ and $h^{1 - \frac{1}{q}} = h^{\frac{1}{p}}$ this completes
  the proof. 
\end{proof}

Observe that the proof of Theorem~\ref{thm:RDE_y2}
only requires Assumption~\ref{as:drift}. In the next theorem we additionally
assume that $x$ is $\frac{1}{p}$-H\"older continuous and $b$ locally Lipschitz
continuous. Then we obtain a more explicit error estimate.

\begin{theorem}[H\"older case]
  \label{thm:RDE_y}
  Let $x \colon [0,T] \to \R^d$ be $\frac{1}{p}$-H\"older continuous for some
  $p \in [1,\infty)$. Suppose that Assumption~\ref{as:drift} and
  Assumption~\ref{as:drift2} are satisfied. Then the implicit Euler scheme
  \eqref{eqn:RN_y} converges to the solution of \eqref{eqn:add_noise} with order
  $\frac{1}{p}$. To be more precise, it holds
  \begin{align*}
    \max_{n \in \{0,\ldots,N_h\}} \big|y(t_n)-y_n\big| 
    \le 2 T \ee^{2C_b T} \Upsilon_b \big( \|y\|_{\infty;[0,T]}\big) 
    \|y\|_{\frac{1}{p};[0,T]} h^{\frac{1}{p}},
  \end{align*}
  for all step sizes satisfying $2 C_b h \leq 1$.
\end{theorem}

\begin{proof}
  In light of Theorem~\ref{thm:RDE_y2} it remains to estimate
  \begin{align*}
    \| b \circ y \|_{p\mathrm{-var};[0,T]}^p =  \sup_{(t_i) \subset [0,T]}
    \sum_i | b(y(t_{i+1})) - b(y(t_{i})) |^p,
  \end{align*}
  where the supremum is taken over all finite partitions of the interval
  $[0,T]$. 
  To this end, let $(t_i)_{i \in \{0,\ldots,N\}}$ be an arbitrary partition. 
  From Corollary~\ref{cor:RDE} it follows that the exact solution $y$ is bounded. 
  Then, it follows from Assumption~\ref{as:drift2} and the H\"older continuity
  of $y$ that
  \begin{align*}
    \sum_i | b(y(t_{i+1})) - b(y(t_{i})) |^p 
    &\le \Upsilon_b \big( \|y\|_{\infty;[0,T]}\big)^p 
    \sum_i | y(t_{i+1}) - y(t_{i}) |^p \\
    &\le \Upsilon_b \big( \|y\|_{\infty;[0,T]}\big)^p
    \|y\|_{\frac{1}{p};[0,T]}^p \sum_i ( t_{i+1} - t_{i} )\\
    &= \Upsilon_b \big( \|y\|_{\infty;[0,T]}\big)^p
    \|y\|_{\frac{1}{p};[0,T]}^p T. 
  \end{align*}
  Inserting this into the error estimate in Theorem~\ref{thm:RDE_y2} then
  yields the assertion.
\end{proof}

\begin{remark}
  \label{rmk:optimal}
  Under additional assumptions on $b$ it is possible to obtain 
  better convergence rates. For instance, in the case of a Brownian driver
  the method \eqref{eqn:RN_y} coincides with an implicit version of the
  Milstein scheme. Provided the vector field $b$ is sufficiently smooth, say $b
  \in \mathcal{C}_b^2$, it is known that the Milstein scheme
  converges pathwise with an order close to $1$. For instance, we refer to the
  standard monographs \cite{KP92, milstein1995, milstein2004}.

  For deterministic and $\frac{1}{p}$-H\"older continuous drivers $x$ the order
  of convergence $\frac{1}{p}$ is in general optimal. One way to improve the
  convergence rates is to artificially randomize the numerical method.
  We point the reader to the more detailed discussions in \cite{EK17,KW2017}
  and the references therein.  
\end{remark}

\section{Multiplicative noise: the rough path case}\label{sec:semiflow}

The multiplicative noise version of the equation \eqref{eqn:add_noise} would (naively) take the form
\begin{align}\label{eqn:mult_noise_ill-posed}
 \diff{y} = b(y)\diff{t} + \sigma(y) \diff x, \quad y(0) = \xi
\end{align}
where $x$ is an $\R^m$-valued path and $\sigma \colon \R^d \to L(\R^m,\R^d)$. However, this equation is ill-posed in the case when $x$ has low (H\"older-) regularity, and we have to use rough path theory to make sense of it. Thus, we look at the rough differential equation
\begin{align}\label{eqn:RDE}
 \diff{y} = b(y)\diff{t} + \sigma(y) \diff \mathbf{x}, \quad y(0) = \xi
\end{align}
instead where $\mathbf{x}$ will be a suitable rough path. Let us first recall the basic notions from rough path theory which can be found e.g. in \cite{FH14}.

\begin{definition}\label{definition:rough_path}
 Let $p \in [2,3)$. A \emph{$\frac{1}{p}$-H\"older rough path} is a pair $\mathbf{x} = (x, \mathbb{X}) \colon [0,T]^2 \to \R^m \oplus (\R^m \otimes \R^m)$ for which the algebraic identy 
 \begin{align}\label{eqn:chen}
  \mathbb{X}_{s,t} - \mathbb{X}_{s,u} - \mathbb{X}_{u,t} = x_{s,u} \otimes x_{u,t}
 \end{align}
 holds for every $s,u,t \in [0,T]$, using the notation $x_{s,t}:=x(t)-x(s)$, and for which
 \begin{align*}
  \| \mathbf{x} \|_{1/p; [0, T]} := \sup_{0 \leq s < t \leq T} \frac{|x_{s,t}|}{|t-s|^{\frac{1}{p}}}  +  \sqrt{  \| \mathbb{X}\|_{2/p; [0, T]} } < \infty,
\end{align*}
where
\begin{align*}
\| \mathbb{X}\|_{2/p; [0, T]} := \sup_{0 \leq s < t \leq T} \frac{|\mathbb{X}_{s,t}|}{|t-s|^{\frac{2}{p}}}.
\end{align*}
Here we write as $ \|\cdot \|_{1/p; [0, T]}$ to distinguish from $ \|\cdot \|_{\frac{1}{p}; [0, T]}$ introduced in Section \ref{sec: pre}. $\mathbf{x}$ is called \emph{geometric} if
\begin{align}
 \operatorname{Sym}(\mathbb{X}_{s,t}) = \frac{1}{2} x_{s,t} \otimes x_{s,t}
\end{align}
holds for every $s,t \in [0,T]$. If $\mathbf{x} = (x, \mathbb{X})$ and $\mathbf{y} = (y, \mathbb{Y})$ are two rough paths, their distance will be measured via the metric
\begin{align*}
 \varrho_{1/p}( \mathbf{x}, \mathbf{y}) := \sup_{0 \leq s < t \leq T} \frac{|x_{s,t} - y_{s,t}|}{|t-s|^{\frac{1}{p}}}  +  \sup_{0 \leq s < t \leq T} \frac{|\mathbb{X}_{s,t} - \mathbb{Y}_{s,t}|}{|t-s|^{\frac{2}{p}}}.
\end{align*}

\end{definition}

It is possible to give a definition of a $p$-variation rough path too, but we will only consider the H\"older case here for simplicity. The object $\mathbb{X}_{s,t}$ for $s<t$ should be thought of the the second iterated integral \mbox{$\int_{s < u < v < t} \diff x_u \otimes \diff x_v$}. Indeed, if $x$ is smooth (e.g. $\frac{1}{p}$-H\"older for $p \in [1,2)$), we can define $\mathbb{X}$ as the second iterated Young integral \cite{You36} and one can show that $(x,\mathbb{X})$ satisfies the conditions stated in Definition \ref{definition:rough_path} and defines a geometric rough path. However, for $p \geq 2$, we are not able to use Young's theory anymore, and we have to \emph{assume} that $\mathbb{X}$ exists. 

\begin{definition}
 Let $x \colon [0,T] \to \R^m$ be a $\frac{1}{p}$-H\"older path. If $W$ is a finite dimensional vector space, a path $y \colon [0,T] \to W$ is called \emph{controlled by $x$} if there exists a $\frac{1}{p}$-H\"older path $y' \colon [0,T] \to L(\R^m,W)$ such that for $R_{s,t} := y_{s,t} - y'(s) x_{s,t}$, one has
 \begin{align*}
  \| R \|_{2/p; [0,T]} := \sup_{0 \leq s < t \leq T} \frac{ |R_{s,t}|}{|t-s|^{\frac{2}{p}}} < \infty.
 \end{align*}
 The path $y'$ is called a \emph{Gubinelli derivative} of $y$. 
\end{definition}
One can show that the space of controlled paths with the norm
 \begin{align*}
  \|y,y'\|_{x,\frac{2}{p};[0,T]} := \|y'\|_{\frac{1}{p}; [0,T]} + \|R\|_{2/p;[0,T]}
 \end{align*}
 is a Banach space. Moreover, if $(y,y')$ is controlled by $x$ and $\sigma$ is a sufficiently smooth function, the path $t \mapsto \sigma(y(t))$ is again controlled by $x$ with Gubinelli derivative $\sigma(y(t))' = D \sigma(y(t))y'(t)$ \cite[Lemma 7.3]{FH14}.

\begin{theorem}
 Let $\mathbf{x} = (x,\mathbb{X}$) be a $\frac{1}{p}$-H\"older rough path and $(y,y')$ be controlled by $x$. Then the rough integral
 \begin{align*}
  \int_0^T y(t) \diff \mathbf{x}(t)
 \end{align*}
 exists as the limit of Riemann sums of the form $\sum y(t_i)x_{t_i,t_{i+1}} + y'(t_i) \mathbb{X}_{t_i,t_{i+1}}$.

\end{theorem}

\begin{proof}
 \cite[Theorem 4.10]{FH14}.
\end{proof}

\begin{definition}
 A path $y \colon [0,T] \to \R^d$ is called a \emph{solution to the rough differential equation}
\begin{align}
 \diff{y} = b(y)\diff{t} + \sigma(y) \diff \mathbf{x}, \quad y(0) = \xi
\end{align}
if
\begin{align*}
 y(t) = \xi + \int_0^t b(y(s))\diff s + \int_0^t \sigma(y(s))\diff \mathbf{x}(s)
\end{align*}
holds for all $t \in [0,T]$ where the second integral makes sense as a rough integral.
\end{definition}

The question whether \eqref{eqn:RDE} possesses a solution will obviously depend on the coefficients. Next, we will define a space of functions $\sigma$ which will be important for us.

\begin{definition}
 For a function $\sigma \colon \R^d \to L(\R^m, \R^d)$, we denote by $D^k \sigma$ the $k$-th Fr\'echet derivative, $k \in \N_0$. The space of all $k$-times Fr\'echet differentiable functions will be denoted by $\mathcal{C}^k_{\text{loc}}$. The space $\mathcal{C}^k_{b}$ consists of all functions $\sigma \in \mathcal{C}^k_{\text{loc}}$ for which all derivatives and the function itself are bounded. Set 
 \begin{align*}
  \mathcal{C}^{\infty}_{\text{loc}} := \bigcap_{k \in \N} \mathcal{C}^{k}_{\text{loc}} \quad \text{and} \quad \mathcal{C}^{\infty}_b := \bigcap_{k \in \N} \mathcal{C}^{k}_b .
 \end{align*}
We will also use the notation
 \begin{align*}
  \| \sigma \|_{\mathcal{C}^k} := \sum_{i = 0}^k \| D^i \sigma \|_{\infty; \R^d}
 \end{align*}
 and
 \begin{align*}
  \| \sigma \|_{\mathcal{C}^k; B} := \sum_{i = 0}^k \| D^i \sigma \|_{\infty; B}
 \end{align*}
 for any subset $B \subset \R^d$.
\end{definition}

In many situations, \eqref{eqn:RDE} defines a continuous flow or at least a semiflow. We recall the definition below.

\begin{definition}
 Let $I \subset \R$ be an interval and $\mathcal{X}$ be a set. A \emph{semiflow} is a map
 \begin{align*}
  \phi \colon \{(s,t)\, :\, s\leq t, \ s,t, \in I \} \times \mathcal{X} \to \mathcal{X}
 \end{align*}
 which satisfies the following two properties:
 \begin{itemize}
  \item[(i)] $\phi(t,t,\cdot) = \operatorname{Id}_{\mathcal{X}}$ for all $t \in I$, where $\operatorname{Id}_{\mathcal{X}}:\mathcal{X} \to \mathcal{X}$ is the identity operator.
  \item[(ii)] $\phi(u,t,\cdot) \circ \phi(s,u,\cdot) = \phi(s,t,\cdot)$ for all $s \leq u \leq t$, $s,u,t \in I$.
 \end{itemize}
 If $\phi \colon I \times I \times \mathcal{X} \to \mathcal{X}$ and property (ii) holds for all $s,t,u \in I$, we call $\phi$ a \emph{flow}. If $\mathcal{X}$ is a topological space and $\phi$ is in addition continuous in all its parameters with respect to the product topology, we call it a \emph{continuous (semi)flow}.

\end{definition}

\begin{proposition}\label{prop:lipschitz_flow}
 Let $\mathbf{x} = (x, \mathbb{X})$ be a $\frac{1}{p}$-H\"older rough path, $p \in [2,3)$, and assume $\| \mathbf{x} \|_{\frac{1}{p};[0,T]} \leq C_0$. 
 Let $b \colon \R^d \to \R^d$ be bounded and Lipschitz continuous and let $\sigma \in \mathcal{C}^{3}_{b}$.
Then \eqref{eqn:RDE} has a unique solution $(y,y')$ in the space of controlled paths with $y'(t) = \sigma(y(t))$. Moreover, the equation induces a continuous flow  $\phi$ and there are constants $\delta > 0$ and $L > 0$ depending on $p$, $C_0$, $T$, $b$ and $\sigma$ such that 
 \begin{align}\label{eqn:loc_lip_estim}
  |\phi(s,t,\zeta_1) - \phi(s,t,\zeta_2)| \leq  (1 + L|t-s|^{\frac{1}{p}}) |\zeta_1 - \zeta_2|
 \end{align}
holds for all $s < t \in [0,T]$ with $|t-s| \leq \delta$ and all $\zeta_1, \zeta_2 \in \R^d$. 
\end{proposition}

\begin{proof}

Proving that \eqref{eqn:RDE} has a unique solution for every initial condition $\xi$ is very similar to \cite[Theorem 8.4]{FH14}: We first define
\begin{align*}
 \mathcal{M}(y,y') := \left( \xi + \int_0^{\cdot} b(y(s))\diff s + \int_0^{\cdot} \sigma(y(s))\diff \mathbf{x}(s) , \sigma(y(\cdot)) \right)
\end{align*}
which is a mapping from the space of controlled paths to itself. Next, one can prove that this mapping is a contraction on a sufficiently small time interval $[0,T_0]$ which yields a fixed point, i.e. a solution. We can repeat the argument on the interval $[T_0,T_1]$, $[T_1,T_2]$ and so on. Using boundedness of $\sigma$ and $b$, the length of these intervals can be bounded from below, therefore we can eventually construct solutions on every given time interval $[0,T]$ by gluing these solutions together.

It remains to prove the estimate \eqref{eqn:loc_lip_estim}. We start with a bound for the solution to \eqref{eqn:RDE}. Set 
 \begin{align}\label{eqn:Rdefine}
  R^{y}_{u,v} := y_{u,v} - y'(u) x_{u,v} = \int_u^v  b(y(w)) \diff w + \int_u^v \sigma(y(w))\, \diff \mathbf{x}(w) - \sigma(y(u))x_{u,v}.
 \end{align}
 We claim that there is a constant $C_1$ depending on the parameters above such that
 \begin{align}\label{eqn:loc_unif_bound_sol}
  \| y,y'\|_{x,\frac{2}{p};[s,t]} = \|y'\|_{\frac{1}{p};[s,t]} + \|R^y\|_{2/p;[s,t]} \leq C_1
 \end{align}
 for all $s<t$ with $|t-s| \leq \delta$. To see this, note first that
 \begin{align*}
  \|y'\|_{\frac{1}{p};[s,t]} = \|\sigma(y)\|_{\frac{1}{p};[s,t]} \leq \| \sigma \|_{\mathcal{C}^1} \|y\|_{\frac{1}{p};[s,t]}.
 \end{align*}
  For the remainder, we have 
 \begin{align*}
  \| R^{y} \|_{2/p ; [s,t]} \leq |t-s|^{1-\frac{2}{p}} \|b\|_{\infty} + \sup_{s \leq u < v \leq t} \frac{\left| \int_u^v \sigma(y(w))\, \diff \mathbf{x}(w) - \sigma(y(u))x_{u,v} \right| }{ |v-u|^{\frac{2}{p}}}.
 \end{align*}
 We can estimate the rough integral using \cite[Theorem 4.10]{FH14}: For $s \leq u \leq v \leq t$,
 \begin{align*}
  \left| \int_u^v \sigma(y(w))\, \diff \mathbf{x}(w) - \sigma(y(u)) x_{u,v} \right| &\leq \left| \int_u^v \sigma(y(w))\, \diff \mathbf{x}(w) - \sigma(y(u)) x_{u,v} - (\sigma(y(u)))' \mathbb{X}_{u,v} \right| + \left| (\sigma(y(u)))' \mathbb{X}_{u,v} \right| \\
   &\leq C \left( \| x \|_{\frac{1}{p};[u,v]} \|R^{\sigma(y)} \|_{2/p;[u,v]} + \| \mathbb{X}\|_{2/p;[u,v]} \|\sigma(y)'\|_{\frac{1}{p};[u,v]} \right) |v-u|^{\frac{3}{p}} \\
   &\quad + \|\sigma\|_{\mathcal{C}^1}  \| \mathbb{X} \|_{2/p;[u,v]} |v-u|^{\frac{2}{p}}
 \end{align*}
 where $C$ depends on $p$ and $T$. We used here that $\sigma(y)' = D \sigma(y) y' = D \sigma(y) \sigma(y)$. Using this identity again, we also obtain
 \begin{align*}
   \|\sigma(y)'\|_{\frac{1}{p};[u,v]} \leq \|\sigma\|_{\mathcal{C}^3} \|y\|_{\frac{1}{p};[u,v]}.
 \end{align*}
 As in \cite[Lemma 7.3]{FH14}, one can show that
 \begin{align*}
  \|R^{\sigma(y)} \|_{2/p;[u,v]} \leq \frac{1}{2} \| \sigma \|_{\mathcal{C}^2} \|y\|_{\frac{1}{p};[u,v]}^2 +  \| \sigma \|_{\mathcal{C}^1} \|R^{y} \|_{2/p;[u,v]}.
 \end{align*}
 Putting these estimates together, we see that there is a constant $C$ depending on the claimed parameters such that
 \begin{align*}
  \| R^{y} \|_{2/p ; [s,t]} \leq C + C|t-s|^{\frac{1}{p}} \|R^{y} \|_{2/p;[s,t]}.
 \end{align*}
 Choosing $\delta = \frac{1}{(2C)^p}$ yields a uniform bound for $\| R^{y} \|_{2/p; [s,t]}$ and thus also for $\| y,y'\|_{x,\frac{2}{p};[s,t]}$ as claimed. We proceed with proving \eqref{eqn:loc_lip_estim}. Let $s \in [0,T]$, $\zeta_1, \zeta_2 \in \R^d$ and 
 \begin{align*}
  y^i(t) = \zeta_i + \int_s^t b(y^i(u))\, \diff u + \int_s^t \sigma(y^i(u))\, \diff \mathbf{x}(u); \quad t \in [s,T]
 \end{align*}
 for $i = 1,2$. We will first give an estimate for the H\"older norm of $R^{y^1} - R^{y^2}$ where 
 \begin{align*}
  R^{y^i}_{u,v} := y^i_{u,v} - (y^i)'(u) x_{u,v} = \int_u^v b(y^i(w)) \diff w + \int_u^v \sigma(y^i(w)) \diff \mathbf{x}(w) - \sigma(y^i(u))x_{u,v}, \quad i = 1,2.
 \end{align*}

 Using the estimate \cite[Theorem 4.10]{FH14} for the rough integral, it is straightforward to show that
 \begin{align}
  \begin{split}\label{eqn:est_diff_R}
  \|R^{y^1} - R^{y^2}\|_{2/p;[s,t]} &\leq \|b\|_{\mathrm{Lip}} \|y^1 - y^2\|_{\infty;[s,t]}|t-s|^{1-\frac{2}{p}} \\
  &\quad + C \left( \|R^{\sigma(y^1)} - R^{\sigma(y^2)} \|_{2/p;[s,t]} + \|\sigma(y^1)' - \sigma(y^2)' \|_{\frac{1}{p};[s,t]} \right) |t-s|^{\frac{1}{p}} \\
  &\quad + C \|D \sigma (y^1)\sigma(y^1) - D \sigma (y^2)\sigma(y^2)\|_{\infty;[s,t]}
  \end{split}
 \end{align}
 where $C$ is a constant depending on the parameters stated above. Clearly,
 \begin{align*}
   \|y^1 - y^2 \|_{\infty;[s,t]} \leq C|\zeta_1 - \zeta_2| + C|t-s|^{\frac{1}{p}} \|y^1 - y^2\|_{\frac{1}{p};[s,t]}.
 \end{align*}
 From \cite[Theorem 7.5]{FH14}, 
 \begin{align*}
   &\|R^{\sigma(y^1)} - R^{\sigma(y^2)} \|_{2/p;[s,t]} +  \|\sigma(y^1)' - \sigma(y^2)' \|_{\frac{1}{p};[s,t]} \\
  \leq\ &C \left( \|R^{y^1} - R^{y^2} \|_{2/p;[s,t]} + \|(y^1)' - (y^2)'\|_{\frac{1}{p};[s,t]} + |(y^1)'(s) - (y^2)'(s)| + |\zeta_1 - \zeta_2| \right)
 \end{align*}
 where we used the uniform bounds obtained in \eqref{eqn:loc_unif_bound_sol}. Next,
 \begin{align*}
  |(y^1)'(s) - (y^2)'(s)| = |\sigma(y^1(s)) - \sigma(y^2(s))| = |\sigma(\zeta_1) - \sigma(\zeta_2)| \leq C|\zeta_1 - \zeta_2|
 \end{align*}
 and
 \begin{align*}
  \|(y^1)' - (y^2)'\|_{\frac{1}{p};[s,t]} \leq C \|y^1 - y^2\|_{\frac{1}{p};[s,t]}.
 \end{align*}
  For the last term in \eqref{eqn:est_diff_R}, we see that
  \begin{align*}
  \|D \sigma(y^1)\sigma(y^1) - D \sigma(y^2)\sigma(y^2)\|_{\infty;[s,t]} \leq C \|y^1 - y^2 \|_{\infty;[s,t]} \leq C |\zeta_1 - \zeta_2| + C|t-s|^{\frac{1}{p}} \|y^1 - y^2\|_{\frac{1}{p};[s,t]}.
  \end{align*}
   From $y^i_{s,t} = \sigma(\zeta_i) x_{s,t} + R^{y^i}_{s,t}$, we also have
 \begin{align*}
  \|y^1 - y^2\|_{\frac{1}{p};[s,t]} \leq C|\zeta_1 - \zeta_2| + |t-s|^{\frac{1}{p}} \|R^{y^1} - R^{y^2}\|_{2/p;[s,t]}.
 \end{align*}
   Putting all these pieces together, we arrive at an estimate of the form 
 \begin{align*}
   &\|y^1 - y^2\|_{\frac{1}{p};[s,t]} + \| R^{y^1} - R^{y^2} \|_{2/p;[s,t]} \\
   \leq\ & C |t-s|^{1-\frac{1}{p}} \|y^1 - y^2 \|_{\frac{1}{p};[s,t]} + C |t-s|^{1-\frac{2}{p}} |\zeta_1 - \zeta_2| + C |\zeta_1 - \zeta_2| \\
   &\quad + C|t-s|^{\frac{1}{p}} \left(  \| R^{y^1} - R^{y^2} \|_{2/p;[s,t]} + \|y^1 - y^2 \|_{\frac{1}{p};[s,t]} + |\zeta_1 - \zeta_2| \right).
 \end{align*}
 Choosing $\delta > 0$ smaller if necessary, we obtain
 \begin{align*}
  \|y^1 - y^2\|_{\frac{1}{p};[s,t]} + \| R^{y^1} - R^{y^2} \|_{2/p;[s,t]} \leq  C |\zeta_1 - \zeta_2|.
 \end{align*}
 Now we have for $t-s \leq \delta$,
 \begin{align*}
  |y^1_t - y^2_t| &\leq  |\zeta_1 - \zeta_2| + |t-s|^{\frac{1}{p}} \|y^1 - y^2\|_{\frac{1}{p};[s,t]}  \leq  (1 + C|t-s|^{\frac{1}{p}}) |\zeta_1 - \zeta_2| \\
 \end{align*}
 which was our claim. 
\end{proof}

Next, we define the schemes we will be interested in. Fix a $\frac{1}{p}$-H\"older rough path $(x,\mathbb{X})$. By Lyons' Extension theorem \cite[Theorem 3.7]{LCL07}, there exists a unique element $\mathbb{X}^3 \colon [0,T]^2 \to (\R^m)^{\otimes 3}$ which satisfies
\begin{align*}
 \| \mathbb{X}^3 \|_{3/p;[0,T]} = \sup_{0 \leq s < t \leq T} \frac{|\mathbb{X}^3_{s,t}|}{|t-s|^{\frac{3}{p}}} < \infty
\end{align*}
and
\begin{align*}
  \mathbb{X}^3_{s,t} - \mathbb{X}^3_{s,u} - \mathbb{X}^3_{u,t} = \mathbb{X}_{s,u} \otimes x_{u,t} + x_{s,u} \otimes \mathbb{X}_{u,t} 
\end{align*}
for every $s,u,t \in [0,T]$. In the sequel, we will often just speak of a $\frac{1}{p}$-H\"older rough path $\mathbf{x} = (x,\mathbb{X}^2,\mathbb{X}^3)$ where we set $\mathbb{X}^2 := \mathbb{X}$.

Note that we can view $\sigma \colon \R^d \to L(\R^m,\R^d)$ as a collection of vector fields $\sigma = (\sigma_1,\ldots,\sigma_m)$ where $\sigma_i \colon \R^d \to \R^d$ for any $i = 1, \ldots, m$. Recall that vector fields are in one-to-one corresponce with first order differential operators: if $V = (V^1,\ldots,V^d)$ is a vector field, the corresponding first order differential operator is defined by 
\begin{align*}
 V \varphi(\zeta) = \sum_{i = 1}^d V^i(\zeta) \partial_i \varphi(\zeta)
\end{align*}
for a differentiable function $\varphi$. If $V$ and $W$ are vector fields, $VW$ denotes the second order differential operator obtained by applying $W$ and $V$ consecutively. 

\begin{definition}\label{definition:impl_schemes}

We fix an equidistant partition $\mathcal{T}^h$ of $[0,T]$ of the form 
\begin{align*}
  \mathcal{T}^h =\{t_0=0<t_1<\dots <t_r <\dots<t_{N_h}=T\}\ \ \mbox{with\
  } t_r = r h
\end{align*}
where the step size $h \in (0,T]$ is determined by $h = \frac{T}{N_h}$, $N_h \in \N$. 
 Let $\mathbf{x} = (x, \mathbb{X}^2, \mathbb{X}^3 )$
 be a $\frac{1}{p}$-H\"older rough path. We define three numerical approximations $(y^l_r)$, $l = 1,2,3$, as follows:
\begin{align}\label{eqn:implicit_Euler}
  y^1_{r+1} = y^1_r + h b(y^1_{r+1}) + \sum_{i = 1}^m \sigma_i(y^1_{r})x^i_{t_r,t_{r+1}}, 
\end{align}

\begin{align}\label{eqn:implicit_Milstein}
 y^2_{r+1} = y^2_r + h b(y^2_{r+1}) + \sum_{i = 1}^m \sigma_i(y^2_{r})x^i_{t_r,t_{r+1}} + \sum_{i,j = 1}^m \sigma_i \sigma_j \operatorname{Id}(y^2_{r}) \mathbb{X}_{t_r,t_{r+1}}^{2;i,j}
\end{align}
resp. 
\begin{align}\label{eqn:implicit_Milstein_order3}
  \begin{split}
 y^3_{r+1} &= y^3_r + h b(y^3_{r+1}) + \sum_{i = 1}^m \sigma_i(y^3_{r})x^i_{t_r,t_{r+1}} + \sum_{i,j = 1}^m \sigma_i \sigma_j \operatorname{Id}(y^3_{r}) \mathbb{X}_{t_r,t_{r+1}}^{2;i,j} \\
 &\quad + \sum_{i,j,k = 1}^m \sigma_i \sigma_j \sigma_k \operatorname{Id}(y^3_r) \mathbb{X}_{t_r,t_{r+1}}^{3;i,j,k}
  \end{split}
\end{align}
for $r \in \{0,\ldots,N_h-1\}$ with initial condition $y^1_0 = y^2_0 = y^3_0 = \xi$, provided solutions to these equations exist and are unique.
\end{definition}

\begin{remark}
 For the readers convenience, we spell out the short-hand notation used above in coordinates: for $\xi \in \R^d$,
 \begin{align*}
  \sigma_i \sigma_j \operatorname{Id}(\xi) &= \sum_{\alpha = 1}^d \sigma_i^{\alpha}(\xi) \partial_{\alpha} \sigma_j(\xi), \\
  \sigma_i \sigma_j \sigma_k \operatorname{Id}(\xi) &= \sum_{\beta = 1}^d \sum_{\alpha = 1}^d \sigma_i^{\beta}(\xi) \partial_{\beta} \sigma_j^{\alpha}(\xi) \partial_{\alpha} \sigma_k(\xi) + \sigma_i^{\beta}(\xi) \sigma_j^{\alpha}(\xi) \partial_{\beta} \partial_{\alpha} \sigma_k(\xi)
 \end{align*}
 where we used the product rule in line 2.

\end{remark}


\begin{theorem}
  \label{thm:well-posedness_mult}
  Let Assumption~\ref{as:drift} be satisfied with one-sided Lipschitz constant $C_b$. 
  If $C_b h < 1$ then there exist unique $\mathbb{R}^d$-valued sequences
  $(y^l_r)_{0 \leq r \leq N_h}$, $l = 1,2,3$, satisfying the difference equations \eqref{eqn:implicit_Euler}, \eqref{eqn:implicit_Milstein} and \eqref{eqn:implicit_Milstein_order3}.
\end{theorem}

\begin{proof}
 As for Theorem \ref{thm:well-posedness}, this is just an application of Proposition \ref{prop:nonlinear}.
\end{proof}

In the next proposition, we calculate the local error of the schemes defined above.

\begin{proposition}\label{prop:1step_estimate}
 Let $\mathbf{x} = (x, \mathbb{X}^2, \mathbb{X}^3)$ be a $\frac{1}{p}$-H\"older rough path, $p \in [1,3)$, and choose $C_0$ such that $\| \mathbf{x} \|_{1/p;[0,T]} \leq C_0$. 
 Let $b \colon \R^d \to \R^d$ be bounded and Lipschitz continuous  with Lipschitz constant $L$ and let $\sigma \in \mathcal{C}^{3}_{b}$. Consider the solution $y$ to the rough differential equation
 \begin{align*}
  \diff{y} = b(y)\diff{t} + \sigma(y) \diff \mathbf{x}, \quad y(t_0) = y_0
 \end{align*}
 with $t_0 \in [0,T]$. For $0 < h < \frac{1}{L}$, set
 \begin{align*}
  y^1_1 = y_0 + h b(y^1_1) + \sum_{i = 1}^m \sigma_i(y_0)x^i_{t_0,t_0 + h}, 
\end{align*}

\begin{align*}
 y^2_1 = y_0 + h b(y^2_1) + \sum_{i = 1}^m \sigma_i(y_0)x^i_{t_0,t_0 + h} + \sum_{i,j = 1}^m \sigma_i \sigma_j \operatorname{Id}(y_0) \mathbb{X}_{t_0,t_0 + h}^{2;i,j}
\end{align*}
resp. 
\begin{align*}
  \begin{split}
 y^3_1 &= y_0 + h b(y^3_1) + \sum_{i = 1}^m \sigma_i(y_0)x^i_{t_0,t_0 + h} + \sum_{i,j = 1}^m \sigma_i \sigma_j \operatorname{Id}(y_0) \mathbb{X}_{t_0,t_0 + h}^{2;i,j} \\
 &\quad + \sum_{i,j,k = 1}^m \sigma_i \sigma_j \sigma_k \operatorname{Id}(y_0) \mathbb{X}_{t_0,t_0 + h}^{3;i,j,k}.
  \end{split}
\end{align*} 
 Then there exists a constant $C > 0$ depending on $p$, $C_0$, $T$, $b$ and $\sigma$ such that
 \begin{align}\label{eqn:onestep_1}
  |y(t_0 + h) - y_1^1| \leq Ch^{\frac{2}{p}}
 \end{align}
 in the case $p \in [1,2)$ and 
 \begin{align}\label{eqn:onestep_2}
  |y(t_0 + h) - y_1^2| \leq Ch^{\left(1 + \frac{1}{p}\right) \wedge \frac{3}{p}}
 \end{align}
 resp.
 \begin{align}\label{eqn:onestep_3}
  |y(t_0 + h) - y_1^3| \leq Ch^{1 + \frac{1}{p}}
 \end{align} 
 in the case $p \in [1,3)$. 

\end{proposition}

\begin{proof}

  To prove \eqref{eqn:onestep_1}, note that
 \begin{align*}
  |y(t_0 + h) - y_1^1| &\leq \left| \int_{t_0}^{t_0 + h} b(y(s)) - b(y(t_0))\, \diff s \right| + h |b(y_1^1) - b(y_0)| \\
  &\quad + \left| \int_{t_0}^{t_0 + h} (\sigma(y(s)) - \sigma(y(t_0))) \diff x(s)  \right|.
 \end{align*}
  For the first integral, we have
 \begin{align*}
  \left| \int_{t_0}^{t_0 + h} b(y(s)) - b(y(t_0))\, \diff s \right| \leq \|b\|_{\mathrm{Lip}} h^{1+\frac{1}{p}} \|y\|_{\frac{1}{p};[t_0,t_0 + h]} \leq C h^{1+\frac{1}{p}} 
 \end{align*}
 where we used that $\|y\|_{\frac{1}{p}}$ can be bounded by a constant $C$ depending on the stated parameters which can be deduced from \eqref{eqn:loc_unif_bound_sol}. For the second term, we use the bound
 \begin{align*}
   |b(y_1^1) - b(y_0)| \leq \|b\|_{\mathrm{Lip}} |y_1^1 - y_0| \leq \|b\|_{\mathrm{Lip}} |b(y_1^1)| h + \|b\|_{\mathrm{Lip}} |\sigma(y_0)| |x(t_0 + h) - x(t_0)| 
  \leq C(h + h^{\frac{1}{p}}).
 \end{align*}
 We can use the standard estimate for Young integrals \cite{You36} for the third term to obtain
  \begin{align*}
  \left| \int_{t_0}^{t_0 + h} (\sigma(y(s)) - \sigma(y(t_0))) \diff x(s)  \right|  \leq C h^{\frac{2}{p}} \|x\|_{\frac{1}{p};[t_0,t_0 + h]}  \|\sigma(y)\|_{\frac{1}{p};[t_0,t_0 + h]} \leq C h^{\frac{2}{p}}.
 \end{align*}
 Hence for a constant $C$,
 \begin{align*}
  |y(t_0 + h) - y_1^1| &\leq C\left(h^{\frac{2}{p}} + h^{1 + \frac{1}{p}} + h^{2} \right)\leq C h^{\frac{2}{p}}
 \end{align*}
 and \eqref{eqn:onestep_1} is shown. We proceed with \eqref{eqn:onestep_2}. By definition, 
 \begin{align*}
  D \sigma(y_0)\sigma(y_0) \mathbb{X}_{t_0,t_0 + h} = \sum_{i,j = 1}^m \sigma_i \sigma_j \operatorname{Id}(y_0)(\mathbb{X}_{t_0,t_0 + h}^{2;i,j}).
 \end{align*}
We have
 \begin{align*}
  |y(t_0 + h) - y_1^2| &\leq \left| \int_{t_0}^{t_0 + h} b(y(s)) - b(y(t_0))\, \diff s \right| + h |b(y_2^1) - b(y_0)| \\
  &\quad + \left| \int_{t_0}^{t_0 + h} \sigma(y(s))\, \diff \mathbf{x}(s)  - \sigma(y(t_0))(x(t_0 + h) - x(t_0)) - D \sigma(y(t_0)) \sigma(y(t_0)) \mathbb{X}_{t_0,t_0 + h} \right|
 \end{align*}
  where
   \begin{align*}
  \left| \int_{t_0}^{t_0 + h} b(y(s)) - b(y(t_0))\, \diff s \right| \leq \|b\|_{\mathrm{Lip}} h^{1+\frac{1}{p}} \|y\|_{\frac{1}{p};[t_0,t_0 + h]}
 \end{align*}
 and
 \begin{align*}
  |b(y_2^1) - b(y_0)| 
  &\leq \|b\|_{\mathrm{Lip}} \|b\|_{\infty} h +  \|b\|_{\mathrm{Lip}} \| \sigma\|_{\infty}  \|x\|_{\frac{1}{p};[t_0,t_0 + h]} h^{\frac{1}{p}} + \|b\|_{\mathrm{Lip}} \|\sigma\|_{\mathcal{C}^1}^2 \|\mathbb{X}\|_{2/p;[t_0,t_0 + h]} h^{\frac{2}{p}} \\
  &\leq C(h + h^{\frac{1}{p}} + h^{\frac{2}{p}}).
 \end{align*}
 It remains to estimate the rough integral. We already saw that $y$ is controlled by $x$ with Gubinelli derivative $y' = \sigma(y)$, and $\sigma(y)$ is controlled by $x$ with Gubinelli derivative $\sigma(y)' = D \sigma(y) \sigma(y)$. By \cite[Theorem 4.10]{FH14},
 \begin{align*}
  &\left| \int_{t_0}^{t_0 + h} \sigma(y(s))\, \diff \mathbf{x}(s)  - \sigma(y(t_0))(x(t_0 + h) - x(t_0)) - D \sigma(y(t_0)) \sigma(y(t_0)) \mathbb{X}_{t_0,t_0 + h} \right| \\
  \leq\ &C h^{\frac{3}{p}} \left( \|x\|_{\frac{1}{p};[t_0,t_0 + h]} \|R^{\sigma(y)} \|_{2/p;[t_0,t_0 + h]} + \|\mathbb{X}\|_{2/p;[t_0,t_0 + h]} \|\sigma(y)\|_{\frac{1}{p};[t_0,t_0 + h]} \right).
 \end{align*}
 For $R^{\sigma(y)}$, note that
 \begin{align*}
  R^{\sigma(y)}_{t_0,t_0 + h} &= \sigma(y(t_0 + h)) - \sigma(y(t_0)) - D \sigma(y(t_0))y'(t_0) x_{t_0,t_0 + h} \\
  &=  \sigma(y(t_0 + h)) - \sigma(y(t_0)) - D \sigma(y(t_0))(y(t_0 + h) - y(t_0)) + D \sigma(y(t_0)) R^y_{t_0,t_0 + h} \\
  &= \frac{D^2 \sigma(\xi)}{2} (y(t_0 + h) - y(t_0))^2 + D \sigma(y(t_0)) R^y_{t_0,t_0 + h}
 \end{align*}
 for some $\xi \in \R^d$ on the line segment between $y(t_0 + h)$ and $y(t_0)$. Therefore,
 \begin{align*}
  \|R^{\sigma(y)} \|_{2/p;[t_0,t_0 + h]} \leq \frac{C}{2} \|y\|_{\frac{1}{p};[t_0,t_0 + h]}^2 + C \|R^y\|_{2/p;[t_0,t_0 + h]}. 
 \end{align*}
 For $R^y$, we have
 \begin{align*}
  R^y_{t_0,t_0 + h} = y(t_0 + h) - y(t_0) -\sigma(y(t_0))x_{t_0 ,t_0 + h} = \int_{t_0}^{t_0 + h} b(y(s)) \diff s + \int_{t_0}^{t_0 + h} \sigma(y(s)) \diff \mathbf{x}(s) - \sigma(y(t_0))x_{t_0 ,t_0 + h}.
 \end{align*}
 Setting
 \begin{align*}
  I_{s,t} := \int_s^t \sigma(y(u)) \diff \mathbf{x}(u)  - \sigma(y(s))x_{s,t} - D \sigma(y(s))\sigma(y(s)) \mathbb{X}_{s,t},
 \end{align*}
  we obtain
 \begin{align*}
  \|R^{y} \|_{2/p;[t_0,t_0 + h]} \leq C h^{1 - \frac{2}{p}} + h^{\frac{1}{p}} \|I\|_{3/p;[t_0,t_0 + h]} + C^2 \|\mathbb{X}\|_{2/p;[t_0,t_0 + h]}.
 \end{align*}
 To summarize, we have seen that there is a constant $C$ depending on the stated parameters such that
 \begin{align*}
  \| I\|_{3/p;[t_0,t_0 + h]} \leq C \|x\|_{\frac{1}{p};[t_0,t_0 + h]} h^{\frac{1}{p}} \| I\|_{3/p;[s,t]} + C.
 \end{align*}
 Therefore, if $h \leq \delta$ with $\delta = \frac{1}{2^p C^p \|x\|_{\frac{1}{p}}^p}$, we obtain the bound  $\| I\|_{3/p;[t_0,t_0 + h]} \leq 2 C$. Using this, we see that for a constant $C$,
 \begin{align*}
  |y(t_0 + h) - y_1^2| &\leq C\left(h^{\frac{3}{p}} + h^{1 + \frac{1}{p}} + h^{1 + \frac{2}{p}} + h^{2} \right)\leq   C h^{\left(1 + \frac{1}{p}\right) \wedge \frac{3}{p}}
 \end{align*}
 provided $h \leq \delta$, and  \eqref{eqn:onestep_2} is shown. The proof for  \eqref{eqn:onestep_2} is conceptually the same. The additional ingredient is a uniform bound for the $\frac{4}{p}$-H\"older norm of
 \begin{align*}
  (s,t) \mapsto \int_s^t \sigma(y(u))\,\diff \mathbf{x}(u)  - \sigma(y(s))x_{s,t} - D \sigma(y(s))\sigma(y(s)) \mathbb{X}_{s,t} - \sum_{i,j,k = 1}^m \sigma_i \sigma_j \sigma_k \operatorname{Id}(y(s))(\mathbb{X}_{s,t}^{3;i,j,k}).
 \end{align*}
 This can be achieved by using \emph{second order Gubinelli derivatives}. A path $\hat{y} \colon [0,T] \to L(\R^m,\R^d)$ is called controlled by the geometric rough path $\mathbf{x}$ with first and second Gubinelli derivatives  
\begin{align*}
 \hat{y}^{(1)} \colon [0,T] &\to L(\R^m \otimes \R^m,\R^d) \cong L(\R^m L(\R^m,\R^d)),\\
 \hat{y}^{(2)} \colon [0,T] &\to L((\R^m)^{\otimes 3},\R^d) \cong L(\R^m, L(\R^m \otimes \R^m, \R^d))
\end{align*}
if $\hat{y}$, $\hat{y}^{(1)}$ and $\hat{y}^{(2)}$ are $\frac{1}{p}$-H\"older continuous and
\begin{align*}
 \hat{y}(t) &= \hat{y}(s) + \hat{y}^{(1)}(s) x_{s,t} + \hat{y}^{(2)}(s) \mathbb{X}^2_{s,t} + R^3_{s,t}, \\
 \hat{y}^{(1)}(t) &= \hat{y}^{(1)}(s) + \hat{y}^{(2)}(s) x_{s,t} + R^{2}_{s,t}
\end{align*}
where $R^k$ is $\frac{k}{p}$-H\"older continuous, $k = 2,3$. This is a special case of the general concept introduced in \cite{Gub10}, see also \cite[Section 7.6]{FH14}. If we set   
\begin{align*}
 J_{s,t} := \hat{y}(s)x_{s,t} + \hat{y}^{(1)}(s) \mathbb{X}^{2}_{s,t} + \hat{y}^{(2)}(s) \mathbb{X}^{3}_{s,t},
\end{align*}
we have
\begin{align*}
 J_{s,t} - J_{s,u} - J_{u,t} &= \left(\hat{y}(s) - \hat{y}(u) + \hat{y}^{(1)}(s) x_{s,u} + \hat{y}^{(2)}(s) \mathbb{X}^2_{s,u} \right)x_{u,t} \\
 &\quad + \left( \hat{y}^{(1)}(s) - \hat{y}^{(1)}(u) + \hat{y}^{(2)}(s) x_{s,u} \right) \mathbb{X}^2_{u,t} + (\hat{y}^{(2)}(s) - \hat{y}^{(2)}(u))\mathbb{X}^3_{u,t} \\
 &= -R^3_{s,u} x_{u,t} - R^2_{s,u} \mathbb{X}^2_{u,t} + (\hat{y}^{(2)}(s) - \hat{y}^{(2)}(u))\mathbb{X}^3_{u,t}.
\end{align*}
Applying the Sewing lemma \cite[Lemma 4.2]{FH14}, we obtain
\begin{align*}
 &\left| \int_s^t \hat{y}(u) \diff \mathbf{x}(u) - \hat{y}(s)x_{s,t} - \hat{y}^{(1)}(s) \mathbb{X}^{(2)}_{s,t} - \hat{y}^{(2)}(s) \mathbb{X}^{(3)}_{s,t} \right| \\
 \leq\ &C|t-s|^{\frac{4}{p}} (\|x\|_{\frac{1}{p};[s,t]} \|R^3\|_{3/p;[s,t]} + \|\mathbb{X}^2\|_{2/p;[s,t]} \|R^2\|_{2/p;[s,t]} + \|\mathbb{X}^3\|_{3/p;[s,t]} \|\hat{y}^{(2)}\|_{\frac{1}{p};[s,t]}),
\end{align*}
similar to \cite[Theorem 4.10]{FH14}. We can apply this estimate to
\begin{align*}
 \hat{y}(u) &= \sigma(y(u)), \\
 \hat{y}^{(1)}(u) &= D \sigma(y(u))\sigma(y(u)) \qquad \text{and} \\
 \hat{y}^{(2)}(u) &= D^2 \sigma(y(u))(\sigma(y(u)) \otimes \sigma(y(u))) + D \sigma(y(u)) D \sigma(y(u)) \sigma(y(u))
\end{align*}
and proceed as above to obtain the estimate
\begin{align*}
 |y(t_0 + h) - y_1^2| &\leq C \left(h^{\frac{4}{p}} + h^{1 + \frac{1}{p}} + h^{1 + \frac{2}{p}} + h^{1 + \frac{3}{p}} + h^{2} \right) \leq C h^{1 + \frac{1}{p}}
\end{align*}
where we used  $1 + \frac{1}{p} < \frac{4}{p}$ for $p < 3$. Details are left to the reader.

\end{proof}

\begin{remark}
 Assuming higher regularity of $b$, one can easily define a modification of the scheme $(y_r^3)$ which has a local error of $\frac{4}{p}$. However, the order of the  implementable schemes which we will define below (cf. Definition  \ref{definition:impl_schemes_simplified}) will not increase for this modification because the rate will be dictated by the rate of the Wong-Zakai approximation, cf. the proof of the forthcoming Theorem \ref{thm:rates_Gaussian_rough}, which is already smaller than the rate obtained for the scheme $(y_r^3)$.
\end{remark}

\begin{theorem}\label{thm:conv_rates}
 Let $\mathbf{x}$ be a $\frac{1}{p}$-H\"older rough path for some $p \in [1,3)$. 
 Let $b \colon \R^d \to \R^d$ satisfy Assumption \ref{as:drift} and \ref{as:drift2} and let $\sigma \in \mathcal{C}^3_{\mathrm{loc}}$. Assume that the rough differential equation \eqref{eqn:RDE} induces a continuous semiflow $\phi$ on the time interval $[0,T]$
  For $h > 0$, consider the numerical approximation $(y^l_n)_{n \in  \{0,\ldots,N_h\}}$, $l = 1,2,3$, defined in  \eqref{eqn:implicit_Euler}, \eqref{eqn:implicit_Milstein} resp. \eqref{eqn:implicit_Milstein_order3}. Then there exist constants $\delta > 0$ and $C > 0$ such that
  \begin{align*}
    \max_{n \in \{0,\ldots,N_h\}} \big|y(t_n)-y^1_n\big| &\le C h^{\frac{2}{p} - 1}
  \end{align*}
  for $p \in [1,2)$ and
  \begin{align*}
    \max_{n \in \{0,\ldots,N_h\}} \big|y(t_n)-y^2_n\big| &\le C h^{\frac{1}{p} \wedge (\frac{3}{p} - 1)}, \\
    \max_{n \in \{0,\ldots,N_h\}} \big|y(t_n)-y^3_n\big| &\le C h^{\frac{1}{p}}
  \end{align*}
  for $p \in [1,3)$ and all step sizes satisfying $h < \delta$.
%

\end{theorem}

\begin{proof}
  Since $\phi$ is continuous, we can find a number $M > 0$ such that
\begin{align*}
 \{ \phi(0,t,\xi)\, :\, t \in [0,T]\} \subset B(0,M).
\end{align*}
  Since $b$ is locally Lipschitz continuous, there is a bounded, Lipschitz continuous function $\bar{b} \colon \R^d \to \R^d$ which coincides with $b$ on $B(0,M)$. Moreover, we can find a $\bar{\sigma} \in \mathcal{C}^3_b$ which coincides with $\sigma$ on $B(0,M)$. Let $\bar{\phi}$ be the flow induced by the rough differential equation
  \begin{align*}
  \diff{y} = \bar{b}(y)\diff{t} + \bar{\sigma}(y) \diff \mathbf{x}.
 \end{align*}
 Let $(\bar{y}^l_n)_{n \in  \{0,\ldots,N_h\}}$, $l = 1,2,3$, denote the numerical approximations defined in  \eqref{eqn:implicit_Euler}, \eqref{eqn:implicit_Milstein} resp. \eqref{eqn:implicit_Milstein_order3} where we replace $b$ by $\bar{b}$ and $\sigma$ by $\bar{\sigma}$. Using the local error obtained in Proposition \ref{prop:1step_estimate} and the Lipschitz property of the flow map $\bar{\phi}$ deduced in Proposition \ref{prop:lipschitz_flow} , it is straightforward, cf. e.g.  \cite[Section 10.3.5]{FV10} or \cite[Proposition 4.1]{RR20}, to deduce the global error bounds
  \begin{align*}
    \max_{n \in \{0,\ldots,N_h\}} \big|\bar{\phi}(0,t_n,\xi) - \bar{y}^1_n\big| &\le C h^{\frac{2}{p} - 1}
  \end{align*}
  for $p \in [1,2)$ and
  \begin{align*}
    \max_{n \in \{0,\ldots,N_h\}} \big|\bar{\phi}(0,t_n,\xi) - \bar{y}^2_n\big| &\le C h^{\frac{1}{p} \wedge (\frac{3}{p} - 1)}, \\
    \max_{n \in \{0,\ldots,N_h\}} \big|\bar{\phi}(0,t_n,\xi) - \bar{y}^3_n\big| &\le C h^{\frac{1}{p}}
  \end{align*}
  for $p \in [1,3)$ for sufficiently small $h > 0$. Since $\bar{\phi}(0,t,\xi) = {\phi}(0,t,\xi)$ for all $t \in [0,T]$, we can choose $\delta > 0$ sufficiently small to obtain that $\bar{y}^l_n \in  B(0,M)$ for every $n \in  \{0,\ldots,N_h\}$, $h < \delta$ and $l = 1,2,3$. From the uniqueness statement in Theorem \ref{thm:well-posedness_mult}, it follows that $\bar{y}^l_n =  {y}^l_n$ for every  $n \in  \{0,\ldots,N_h\}$, $h < \delta$ and $l = 1,2,3$ which shows the claim.

\end{proof}

\begin{remark}
 At the current stage, we do not know whether Assumption \ref{as:drift} and \ref{as:drift2} on $b$ alone imply the existence of a semiflow for a generic rough path $\mathbf{x}$, even for $\sigma$ being bounded. In \cite{RS17}, one of us together with M.~Scheutzow formulated a further condition: We assumed that there exists a constant $C > 0$ such that
  \begin{align}\label{eqn:normal_growth}
   \left| b(\xi) - \frac{\langle b(\xi),\xi \rangle \xi}{|\xi|^2} \right| \leq C (1 + |\xi|) \quad \text{for all } \xi \in \R^d \setminus \{0\}.
  \end{align}
 Assuming this assumption in addition to \ref{as:drift} and \ref{as:drift2}, \cite[Theorem 4.3]{RS17} implies the existence of a semiflow provided $\sigma \in \mathcal{C}^4_b$, therefore Theorem \ref{thm:conv_rates} applies in this case. The subtle case of unbounded diffusion vector fields was discussed by Lejay in the two works \cite{Lej09,Lej12}. 
\end{remark}

We want to apply numerical schemes in the stochastic case now, i.e. when the driving rough path is random. In this context, the higher order objects (i.e. the iterated integrals) are usually not explicitly known and hard to simulate. To overcome this issue, Deya-Neuenkirch-Tindel propose in \cite{DNT12} a numerical scheme in which they replace the higher order objects by products of increments of the path. The same idea motivates us to look at the following schemes:
\begin{definition}\label{definition:impl_schemes_simplified}
Let $\mathcal{T}^h$ denote the partition 
\begin{align*}
  \mathcal{T}^h =\{t_0=0<t_1<\dots <t_r <\dots<t_{N_h}=T\}\ \ \mbox{with\
  } t_r = rh
\end{align*}
with step size $h = \frac{T}{N_h}$, $N_h \in \N$. Let $x \colon [0,T] \to \R^m$ be a path. Then we define two numerical schemes $(y^{\mathfrak{s},l}_r)$, $l = 2,3$, as follows:

\begin{align}\label{eqn:implicit_Milstein_simple}
 y^{\mathfrak{s},2}_{r+1} = y^{\mathfrak{s},2}_r + h b(y^{\mathfrak{s},2}_{r+1}) + \sum_{i = 1}^m \sigma_i(y^{\mathfrak{s},2}_{r})x^i_{t_r,t_{r+1}} + \frac{1}{2} \sum_{i,j = 1}^m \sigma_i \sigma_j \operatorname{Id}(y^{\mathfrak{s},2}_{r}) x^i_{t_r,t_{r+1}} x^j_{t_r,t_{r+1}}
\end{align}
resp. 
\begin{align}\label{eqn:implicit_Milstein_order3_simple}
  \begin{split}
 y^{\mathfrak{s},3}_{r+1} &= y^{\mathfrak{s},3}_r + h b(y^{\mathfrak{s},3}_{r+1}) + \sum_{i = 1}^m \sigma_i(y^{\mathfrak{s},3}_{r})x^i_{t_r,t_{r+1}} + \frac{1}{2} \sum_{i,j = 1}^m \sigma_i \sigma_j \operatorname{Id}(y^{\mathfrak{s},3}_{r}) x^i_{t_r,t_{r+1}} x^j_{t_r,t_{r+1}} \\
 &\quad + \frac{1}{6} \sum_{i,j,k = 1}^m \sigma_i \sigma_j \sigma_k \operatorname{Id}(y^{\mathfrak{s},3}_r) x^i_{t_r,t_{r+1}} x^j_{t_r,t_{r+1}} x^k_{t_r,t_{r+1}}
  \end{split}
\end{align}
for $r \in \{0,\ldots,N_h-1\}$ with initial condition $y^{\mathfrak{s},2}_0 = y^{\mathfrak{s},3}_0 = \xi$, provided solutions to these equations exist and are unique.
\end{definition}

 We have already seen that Assumption~\ref{as:drift} for be $b$ implies that these schemes are well-defined  provided $C_b h < 1$.

We will apply the schemes to rough differential equations driven by Gaussian processes in the sense of Friz-Victoir \cite{FV10-2}. Next, we recall a basic existence theorem.

\begin{theorem}\label{thm:existence_Gaussian}
  Let $X = (X^1, \ldots, X^m)$ be a continuous, centered Gaussian process with independent components. Assume that each component has stationary increments and that the function $\varrho^2$ given by
 \begin{align*}
  \varrho^2(|t-s|) = \E( |X_t - X_s|^2)
 \end{align*}
  is concave with $\varrho(\tau) = \mathcal{O}(\tau^{\frac{1}{\rho}})$ for $\tau \to 0$ and some $\rho \in \big[1,\frac{3}{2}\big)$. Then there exists a lift of $X$ to an enhanced Gaussian process $\mathbf{X} = (X,\mathbb{X})$ on a set of full measure, i.e. $\mathbf{X}$ is almost surely a $\frac{1}{p}$-H\"older rough path for any $2 \rho < p < 3$. The second order process $\mathbb{X}$ is given as a limit in probability of usual Riemann sums.
\end{theorem}

\begin{proof}
 Cf. \cite{FGGR16} or \cite[Theorem 10.9]{FH14}.
\end{proof}

\begin{theorem}\label{thm:rates_Gaussian_rough}
 Let $X = (X^1, \ldots, X^m)$ be as in Theorem \ref{thm:existence_Gaussian} with corresponding lift $\mathbf{X}$. Assume that 
  $b$ satisfies Assumption \ref{as:drift} and \ref{as:drift2} and that $\sigma \in \mathcal{C}^{\infty}_{\text{loc}}$. Assume that for every given $\frac{1}{p}$-H\"older rough path $\mathbf{x}$, the rough differential equation \eqref{eqn:RDE} induces a continuous semiflow $\phi^{\mathbf{x}}$ for which
  \begin{align}\label{eqn:unif_bound_flow}
   \sup_{\mathbf{x}\, :\, \|\mathbf{x}\|_{1/p} \leq C_0} \sup_{t \in [0,T]} |\phi^{\mathbf{x}}(0,t,\xi)| < \infty
  \end{align}
  for any given $C_0 > 0$.  Let $Y$ denote the solution to the random rough differential equation \eqref{eqn:RDE} where we replace $\mathbf{x}$ by $\mathbf{X}$ and let $Y^{\mathfrak{s},l}_n$, $l = 2,3$, denote the corresponding numerical approximation defined in \eqref{eqn:implicit_Milstein_simple} and \eqref{eqn:implicit_Milstein_order3_simple}.

  Then for every $2\rho < p < 3$, there are almost surely finite random variables $\delta$, $C_2$ and $C_3$ such that 
    \begin{align*}
    \max_{n \in \{0,\ldots,N_h\}} \big|Y(t_n) - Y^{\mathfrak{s},2}_n\big| &\le C_2 h^{\frac{3}{p} - 1}, \\
    \max_{n \in \{0,\ldots,N_h\}} \big|Y(t_n) - Y^{\mathfrak{s},3}_n\big| &\le C_3 h^{\frac{2}{p} - \frac{1}{2}}
  \end{align*}
  for all step sizes satisfying $h < \delta$.

\end{theorem}

\begin{remark}
 The assumption \eqref{eqn:unif_bound_flow} is very natural in the context of rough differential equations. It is satisfied, for instance, if the vector fields are bounded \cite[Proposition 8.3]{FH14} or if $\sigma \in \mathcal{C}^{\infty}_b$ and $b$ satisfies \eqref{eqn:normal_growth} \cite{RS17}.
\end{remark}

\begin{proof}[Proof of Theorem \ref{thm:rates_Gaussian_rough}]

  The idea of the proof is from \cite{DNT12} and was also used in \cite{FR14}.
 First, it is easily seen that the schemes defined
 in \eqref{eqn:implicit_Milstein_simple} resp.
 \eqref{eqn:implicit_Milstein_order3_simple} coincide with the ones defined in
 \eqref{eqn:implicit_Milstein} resp. \eqref{eqn:implicit_Milstein_order3} when
 $\mathbf{X}$ is replaced by the canonical lift of the process $X^{h}$ which is defined as the piecewise linear approximation of
 $X$ at the points given by $\mathcal{T}^h$. Thus,
 \begin{align}\label{eqn:wong_zakai_decomp}
  \max_{n \in \{0,\ldots,N_h\}} \big|Y(t_n) - Y^{\mathfrak{s},l}_n\big| \leq \sup_{t \in [0,T]} |Y(t) - Y^h(t)| + \max_{n \in \{0,\ldots,N_h\}} |Y^h(t_n) - Y^{h;l}_n|
 \end{align}
 where $Y^h$ is the solution to
 \begin{align*}
  \diff Y^h = b(Y^h) \diff t + \sigma(Y^h) \diff X^h(\omega); \quad Y_0 = \xi
 \end{align*}
 and $(Y^{h;l}_n)$, $l = 1,2$, is defined as in Definition \ref{definition:impl_schemes} when the rough path is the canonical lift of $X^h$. With Theorem \ref{thm:conv_rates}, we can give an estimate for the second term on the right hand side of the inequality \eqref{eqn:wong_zakai_decomp} and obtain the rates $\frac{3}{p} - 1$ resp. $\frac{1}{p}$. These estimates are indeed uniform due to assumption \eqref{eqn:unif_bound_flow}. The first term in \eqref{eqn:wong_zakai_decomp} is the Wong-Zakai error. In the case of bounded vector fields, the solution map of a rough differential equation is locally Lipschitz continuous in the rough path topology \cite[Theorem 8.5]{FH14}. Using a localization argument as in Theorem \ref{thm:conv_rates} together with assumption \eqref{eqn:unif_bound_flow}, we may assume that the map is locally Lipschitz continuous also in our case. We can thus apply the results in \cite{FR14} to obtain a rate of $\frac{2}{p} - \frac{1}{2}$ for the Wong-Zakai approximation. Since $\frac{3}{p} - 1 < \frac{2}{p} - \frac{1}{p}$ and $\frac{2}{p} - \frac{1}{2} < \frac{1}{p}$ for $p > 2$, the claim follows.

\end{proof}

\section{Numerical experiments}
\label{sec:numexp}

In this section we perform several numerical experiments with the numerical
methods discussed in this paper. In our examples we focus on rough differential
equations where the driver is generated by a fractional Brownian motion.

\begin{example}
  In the following we investigate the performance of the implicit Euler method
  \eqref{eqn:RN_y} applied to a scalar rough differential equation driven by an
  additive fractional Brownian motion with different regularities. 
  To be more precise, we consider
  \begin{align}
    \label{eq:RDEexample1}
    \begin{split}
      \begin{cases}
        \diff y(t) = (y(t)-y^3(t)) \diff{t} + \diff{B^H(t)}, \quad t \in (0,1],\\
        y(0) = -3.0,
      \end{cases}
    \end{split}
  \end{align}
  where $B^H$ is a real-valued fractional Brownian motion with Hurst parameter
  $H\in (0,1)$. Note that the function $b \colon \R \to \R$ defined by
  $b(y):=y-y^3$ for $y \in \R$ satisfies a one-sided
  Lipschitz condition with constant $C_b = 1$. In particular,
  Assumptions~\ref{as:drift} and \ref{as:drift2} are satisfied. 
  In the experiment, we choose $H$ to be $0.75$, $0.5$, $0.25$ and $0.10$ respectively.
  For the simulation of the numerical scheme \eqref{eqn:RN_y} 
  we consider the step sizes $h\in \{2^{-7},2^{-8}, 2^{-9},2^{-10},2^{-11},2^{-12}\}$ 
  and compare them to a reference solution obtained via a finer step size of
  $h_{\mathrm{ref}} = 2^{-14}$.   

  The fractional Brownian motion, as a Gaussian noise, is fully characterized
  by its mean and covariance function. For our numerical experiment we first
  simulate a path of the fractional Brownian motion on the time grid
  \begin{align*}
    \mathcal{T}^{h_{\mathrm{ref}}}
    = \{t_0=0<t_1<\ldots<t_j<\ldots <t_{N_{h_{\mathrm{ref}}}}
    =1\}\ \mbox{with\ } t_j=j h_{\mathrm{ref}},
  \end{align*}
  i.e., with the reference step size $h_{\mathrm{ref}}$.
  Since the increments of a fractional Brownian motion are, in general, not
  mutually independent we generate the full path at once. To this end
  we first compute the
  $N_{h_{\mathrm{ref}}}\times N_{h_{\mathrm{ref}}}$-dimensional
  covariance matrix $C^H$ whose $(i,j)$-th entry is defined by
  \begin{align*}
    C^H_{(i,j)}
    =\mathbb{E}[B^H(t_i)B^H(t_j)]=\frac{1}{2}\big(|t_i|^{2H}+ 
    |t_j|^{2H}-|t_i-t_j|^{2H}\big),\ \mbox{for\
    }i,j\in\{1,\ldots,N_{h_{\mathrm{ref}}}\}. 
  \end{align*}
  Note that the covariance matrix $C^H$ is positive definite and symmetric. 
  Thus, by an application of the Cholesky decomposition 
  we obtain a lower-triangular matrix $L \in \R^{N_{h_{\mathrm{ref}}}\times
  N_{h_{\mathrm{ref}}}}$ with $L L^{\top} = C^H$. 
  Then we draw from the distribution of a sample path of the fractional
  Brownian motion by taking note of
  \begin{align*}
    \big(B^H(t_1), \ldots, B^H(t_{N_{h_{\mathrm{ref}}}})\big)^{\top}
    \sim L V,
  \end{align*}
  where $V = (V_1,\ldots,V_{N_{h_{\mathrm{ref}}}})^{\top}$
  is an $N_{h_{\mathrm{ref}}}$-dimensional standard normally distributed
  vector. For the simulation with larger step sizes we simply restrict 
  the generated trajectory of $B^H$ to the coarser time grid.

  Once the trajectory of the fractional Brownian motion is simulated we can
  directly implement the implicit Euler method \eqref{eqn:RN_y} for the
  approximation of the rough differential equation \eqref{eq:RDEexample1}. In
  each step of the method we have to solve a nonlinear equation. In our
  experiment we accomplished this by an application of Newton's method.
  
  In Figure~\ref{fig1} we show the experimental pathwise errors for the
  different values of the Hurst parameter $H$. The plots show the 
  errors against the underlying step size, i.e., the number $n$
  on the $x$-axis indicates the corresponding simulation is based on the step 
  size $h = 2^{-n}$.  

  First, we observe that all four curves become seemingly less smoother when
  the value of the Hurst parameter decreases. This is expected from the
  decreasing smoothness of the driving path $B^H$.
  In order to compare the convergence behaviour of the implicit Euler method
  \eqref{eqn:RN_y} in our experiments with the theoretical
  result in Theorem~\ref{thm:RDE_y} recall that the path of a fractional
  Brownian motion is $\alpha$-H\"older continuous for every $\alpha \in (0,H)$.
  Thus, the theoretical order of convergence obtained in 
  Theorem~\ref{thm:RDE_y} is essentially equal to $H$. 
  The respective theoretical orders of convergence are indicated by
  the order lines in each plot in Figure~\ref{fig1}. Comparing this with the
  actually observed errors in our experiment
  we conclude that the performance of the implicit
  Euler method is apparently better in this example 
  than predicted by Theorem~\ref{thm:RDE_y}. We also mention that, although
  Figure~\ref{fig1} only shows the result for just one particular sample path,
  one essentially obtains the same qualitative behavior of the numerical error
  for other typical paths of the driving fractional Brownian motion.
  
  \begin{figure}[t]
    \centering
    \subfigure[a][$H=0.75$]{\includegraphics[width=0.4\textwidth]{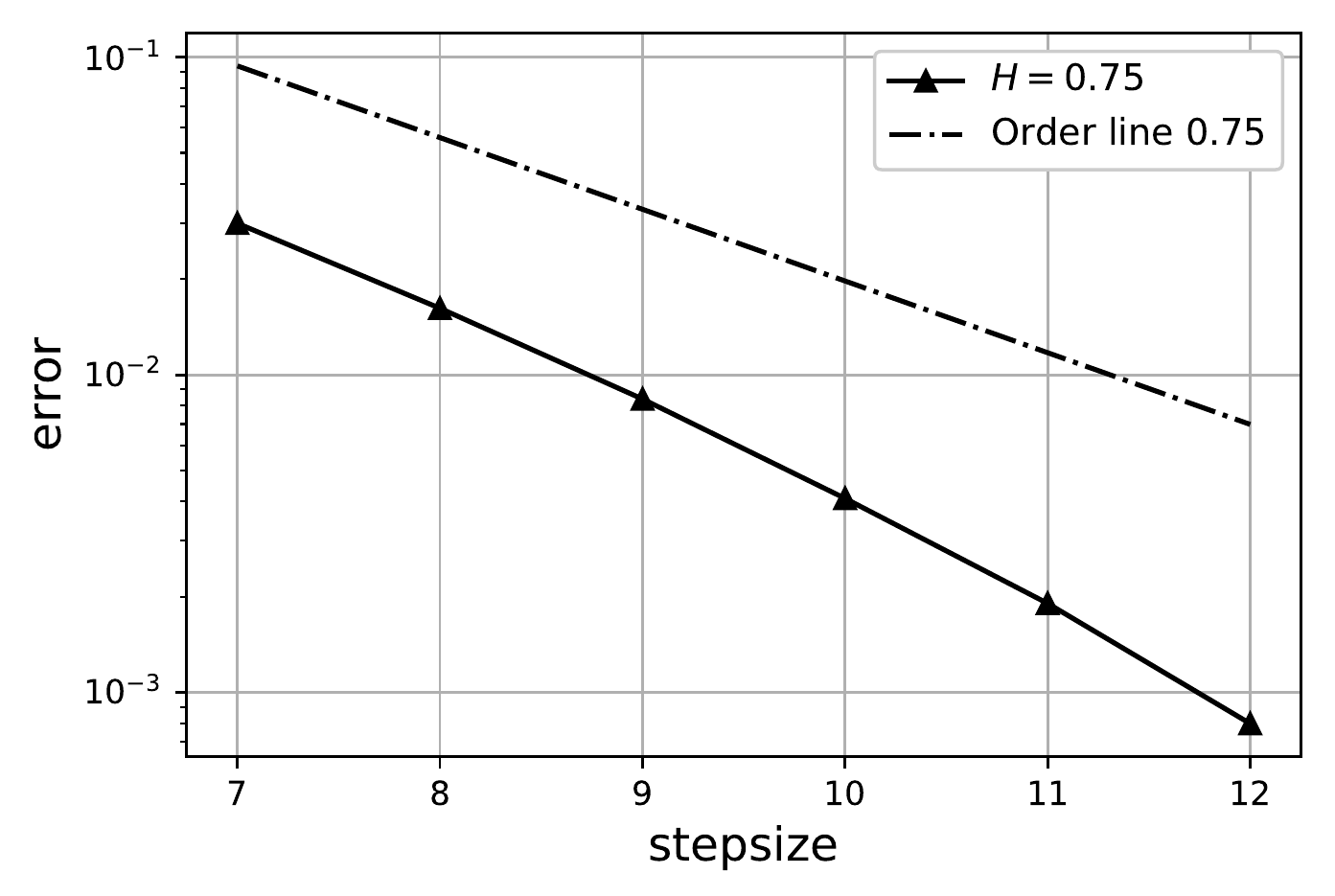}}
    \qquad
    \subfigure[b][$H=0.50$]{\includegraphics[width=0.4\textwidth]{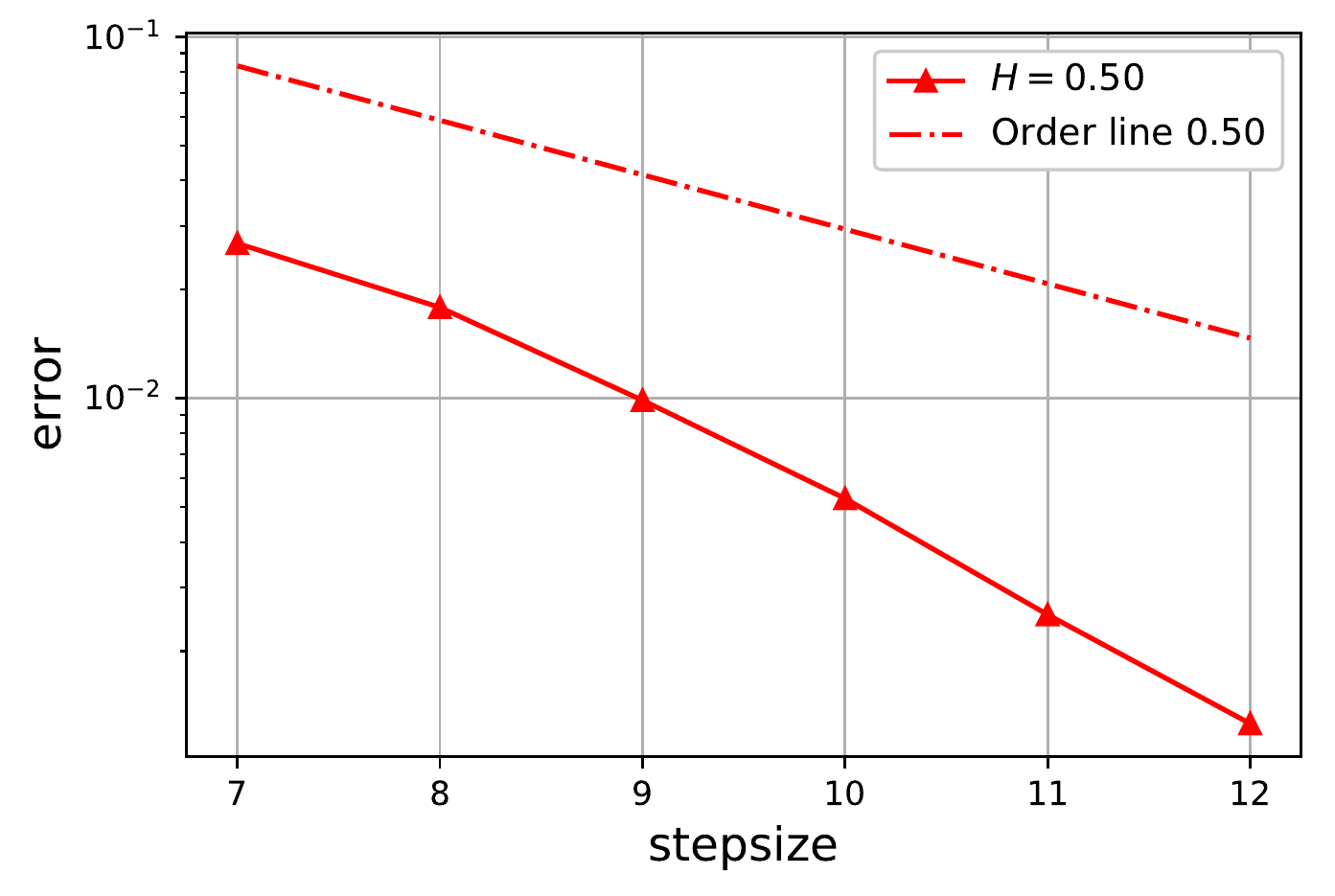}}\\
    \subfigure[c][$H=0.25$]{\includegraphics[width=0.4\textwidth]{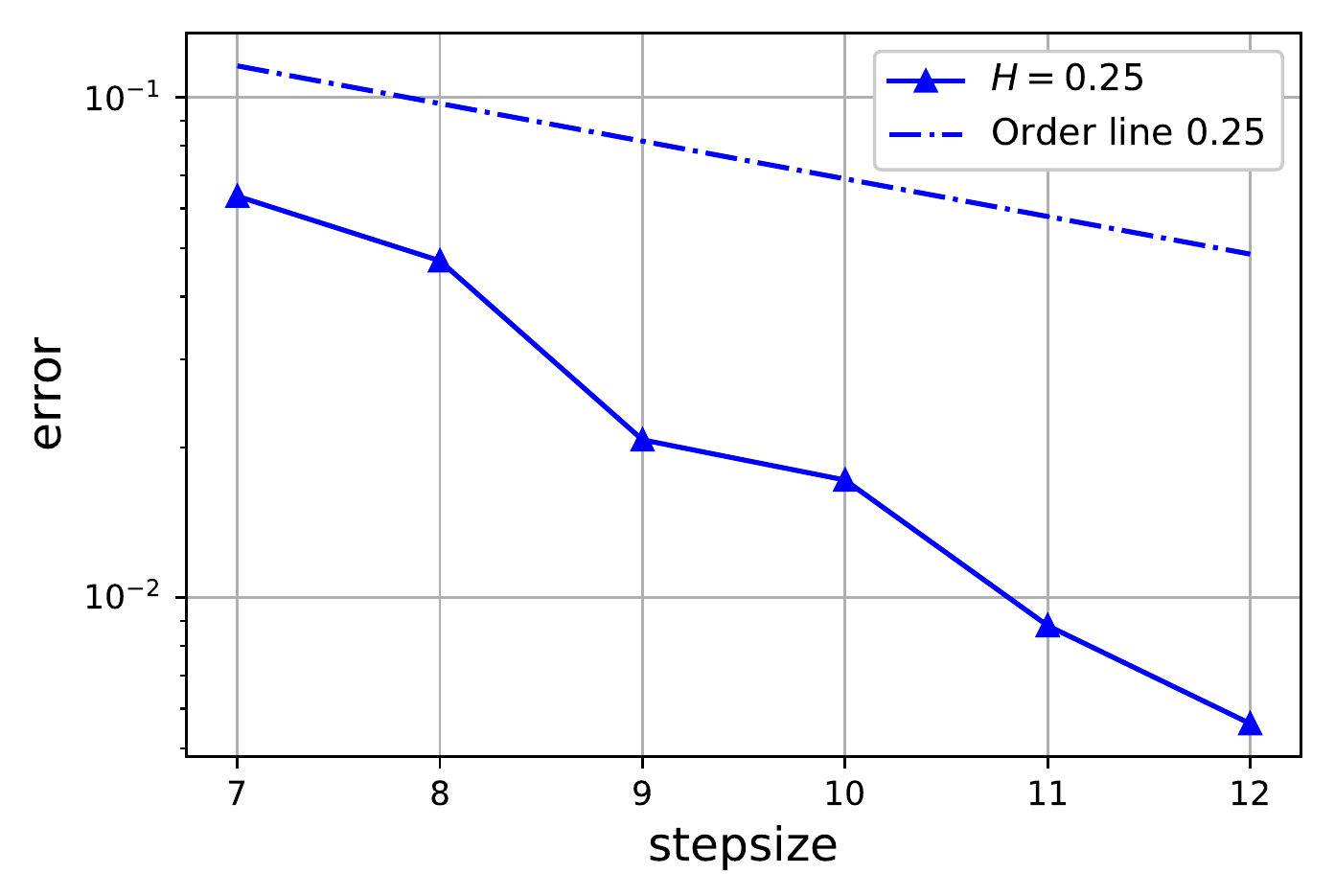}}
    \qquad
    \subfigure[d][$H=0.10$]{\includegraphics[width=0.4\textwidth]{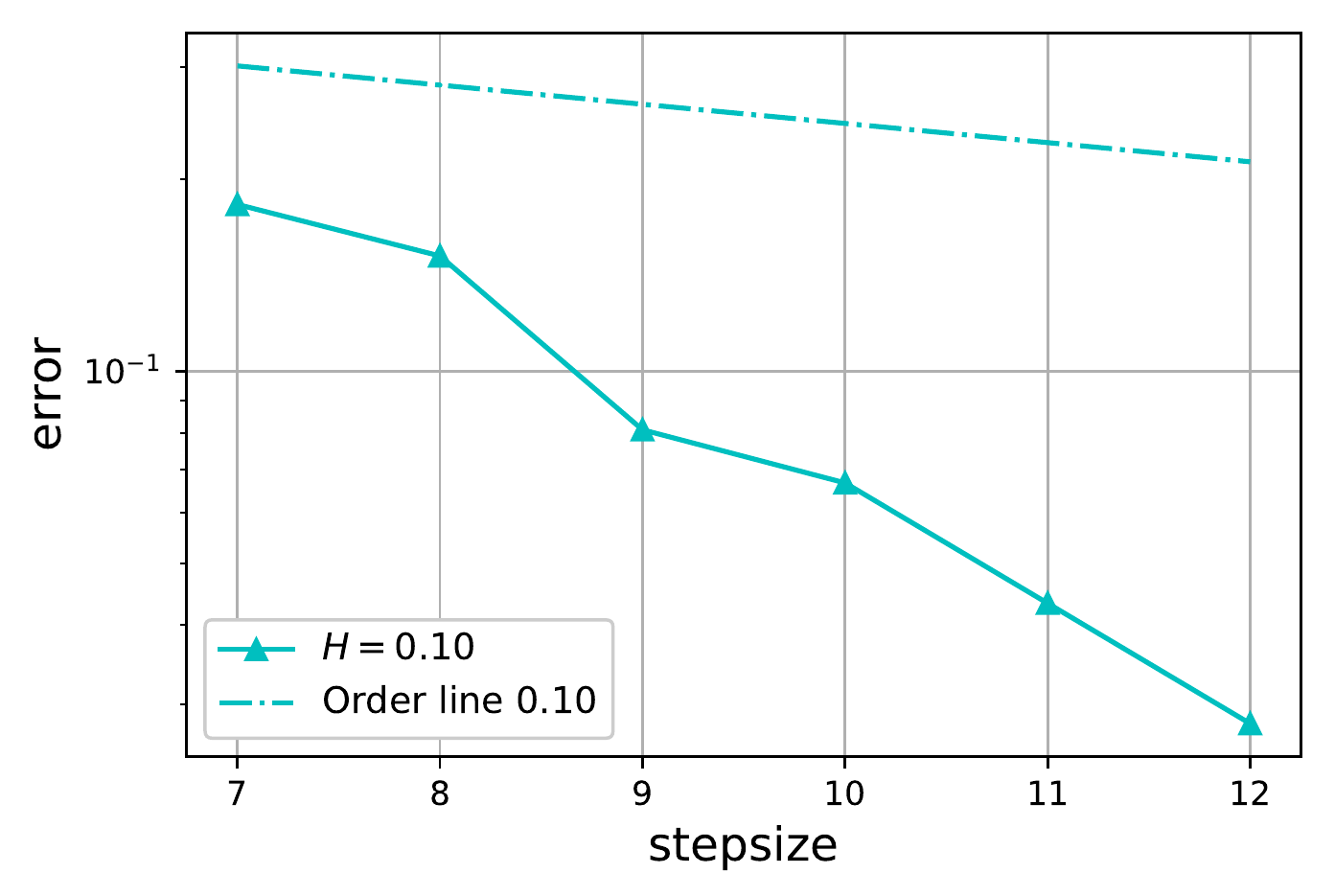}}
    \caption{Numerical experiments for the implicit Euler method for the RDE
    \eqref{eq:RDEexample1} with $H \in \{\frac{1}{10}, \frac{1}{4}, \frac{1}{2},
    \frac{3}{4}\}$: step sizes versus pathwise errors.}  
    \label{fig1}
  \end{figure}

  \begin{table}[h]
    \caption{   \label{tab:randGFEM_L2_err}
    Numerical values of the pathwise errors and experimental order of
    convergence (EOC) for implicit Euler method for the RDE
    \eqref{eq:RDEexample1} simulations.}  
    \begin{tabular}{p{1.2cm}|p{1.5cm}p{1.0cm}|p{1.5cm}p{0.9cm}|p{1.5cm}p{0.9cm}|p{1.5cm}p{0.9cm}}
     & $H=0.75$ &     & $H=0.50$   &  & $H=0.25$ & &  $H=0.10$ &  \\ 
     \noalign{\smallskip}\hline\noalign{\smallskip}
     $h$     & error   & EOC     & error & EOC     & error & EOC& error & EOC \\ 
     \noalign{\smallskip}\hline\noalign{\smallskip}
       0.007813  & 0.029995  & & 0.026809&  & 0.063482 &   & 0.182718&\\ 
       0.003906  & 0.016201 & 0.88 & 0.017836 & 0.60  & 0.047149 & 0.43 & 0.151680&0.27\\ 
      0.001953  & 0.008391  & 0.95 & 0.009873 & 0.85  & 0.020681 & 1.18 & 0.080940&0.90\\ 
      0.000977  & 0.004081   & 1.04 & 0.005284 & 0.90  & 0.017187& 0.27& 0.066802 &0.28\\ 
      0.000488  & 0.001907   & 1.10 & 0.002523 & 1.06  & 0.008794& 0.97&0.043243 &0.63\\ 
      0.000244  & 0.000798   & 1.25 & 0.001261 & 1.00  & 0.005601& 0.65& 0.027993&0.63\\ 
        \noalign{\smallskip}\hline\noalign{\smallskip}
       Average  & & 1.04 &  &0.88  &  &0.70  & &0.54\\
    \end{tabular}
  \end{table}

  Table~\ref{tab:randGFEM_L2_err} contains the numerical values of the computed
  errors displayed in Figure~\ref{fig1}. In addition, we computed the
  corresponding experimental orders of convergence defined by 
  \begin{equation*}
    \mbox{EOC} = \dfrac{ \log(\mbox{error}(2^{-i})) -
    \log(\mbox{error}(2^{-i+1})) }{ \log(2^{-i})-\log(2^{-i+1})}  
  \end{equation*}
  for $i\in \{7,8,9,10,11,12\}$, where the term $\mbox{error}(2^{-i})$ denotes
  the error with step size $2^{-i}$. Although the experimental orders of 
  convergence are better then predicted by Theorem~\ref{thm:RDE_y}
  the absolute values of the errors are visibly influenced by the
  Hurst parameter.

  Finally, let us remark that for $H = 0.5$ (standard Brownian motion)
  the order of convergence observed in our experiment is close to $1$. 
  This is in line with standard results for stochastic differential equations 
  with additive noise since in this case the implicit Euler method coincides
  with a Milstein-type method. 
  Regarding the optimality of the convergence rates we also refer to
  the discussion in Remark \ref{rmk:optimal}.  
\end{example}

\begin{example} 
  In the second example we consider the following scalar rough
  differential equation driven by an additive fractional Brownian motion $B^H$
  with Hurst parameter $H=0.75$,  
  \begin{align}
    \label{eq:RDEexample3}
    \begin{split}
      \begin{cases}
        \diff{y(t)} = -70y(t)\diff{t} + \diff{B^H(t)}, \quad t \in (0,1],\\
        y(0) = 2.7.
      \end{cases}
    \end{split}
  \end{align}
  It can be verified that $b\colon \R \to \R$ defined by $\R \ni z\mapsto -70z
  \in \R$ is a one-sided Lipschitz function with constant $-70$, while it
  is globally Lipschitz continuous with constant $70$. This discrepancy 
  renders the problem \emph{stiff}. This usually has the effect that the
  implicit Euler method has a much less restrictive upper step size bound
  compared to its explicit counterpart. For a more formal introduction of
  \emph{stiffness} for numerical methods we refer to \cite{HW96}.

  We can easily illustrate the difference in the stability behavior between the
  explicit and the implicit Euler method in light of the equation
  \eqref{eq:RDEexample3}. First observe that the upper step size limits for the
  implicit Euler method \eqref{eqn:RN_y} in Theorem~\ref{thm:well-posedness}
  and Theorem~\ref{thm:RDE_y2} are void for every
  non-positive one-sided Lipschitz constants, since $C_b h < 1$ holds then true for
  any step size $h = \frac{T}{N_h}$.  

  Next, let us recall that explicit Euler method is given by
  \begin{align}
    \label{eqn:RN_explicit}
    \begin{split}
      y^e_{j+1}&=y^e_{j}+hb(y^e_{j})+B^H_{j+1}-B^H_{{j}}\mbox{\ \ for\ } j\in
      \{0,\cdots,N_h-1\}, 
    \end{split}
  \end{align}
  with $y^e_0=2.7$. Observe that the drift function $b$ in \eqref{eq:RDEexample3}
  gives a strong push towards the origin. However, this behavior is only
  reproduced by the explicit Euler method if the step size is sufficiently
  small. To be more precise, 
  the one-step map of the explicit Euler method is estimated by
  \begin{align*}
    \begin{split}
      &\big|y_{j+1}^e\big| \le \big|y^e_{j}+h b(y^e_{j}) \big|
      + \big|B^H_{j+1} - B^H_{{j}}\big| 
    \end{split}
  \end{align*}
  Thus, the drift part of the explicit Euler method is a contraction if and
  only if
  \begin{align}
    \label{eqn:1dmap}
    \big| 1 + h b'(y_s) \big| = |1-h 70| \leq 1.
  \end{align}
  Compare further with the linear asymptotic stability of dynamical systems in
  discrete time, for instance, in \cite[Chapter 10]{S18}.
  One easily verifies that \eqref{eqn:1dmap} leads to the step size bound
  $h \le h_o := \frac{1}{35} < 2^{-5}$. If this bound is violated then the drift
  part of the explicit Euler method is too negative and typical trajectories of the
  explicit method start to oscillate. 
  
  \begin{figure}[t]
    \centering
    \subfigure[a][Implicit scheme]{
      \includegraphics[width=0.45\textwidth]{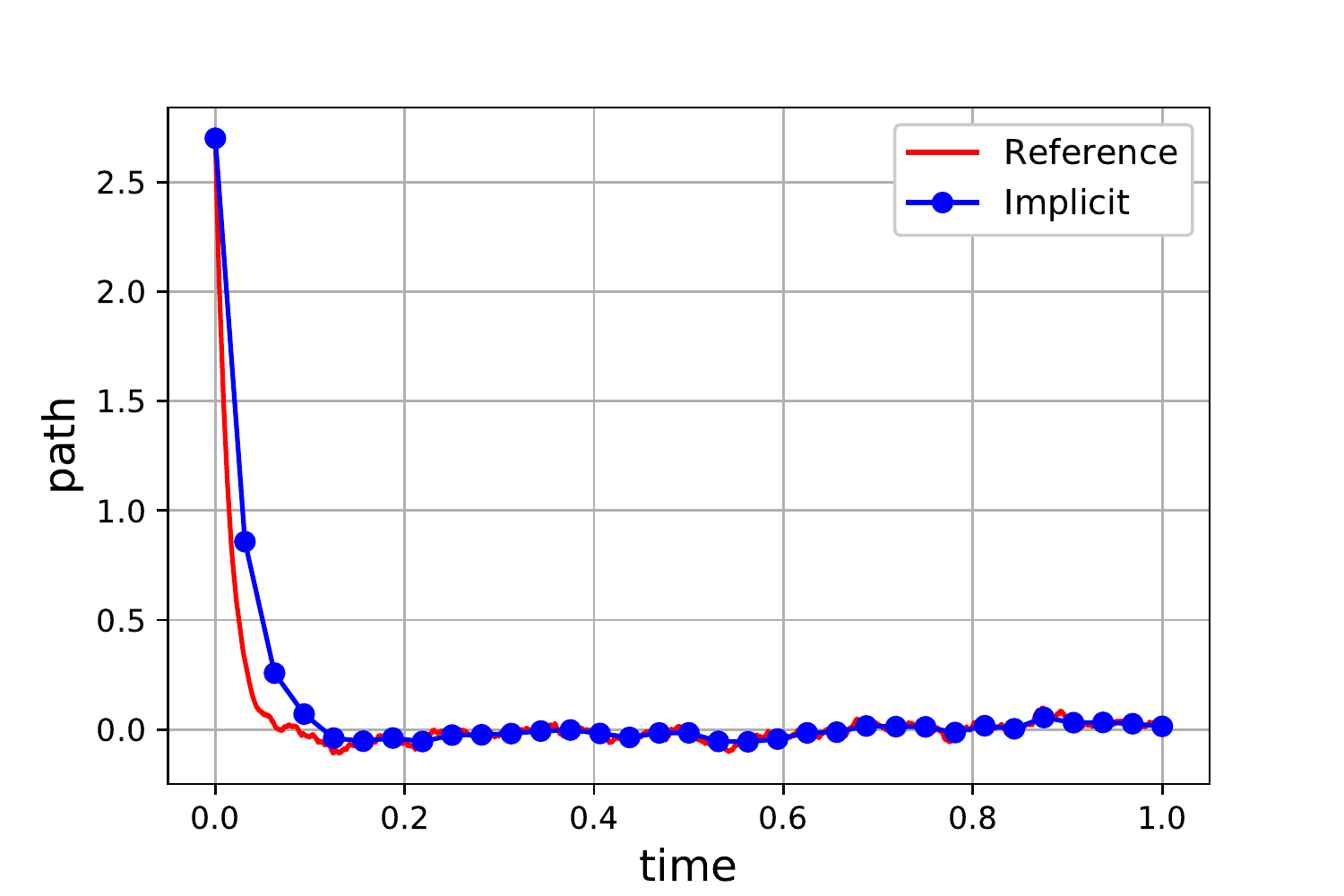}
    }
    \qquad
    \subfigure[b][Explicit scheme]{
      \includegraphics[width=0.45\textwidth]{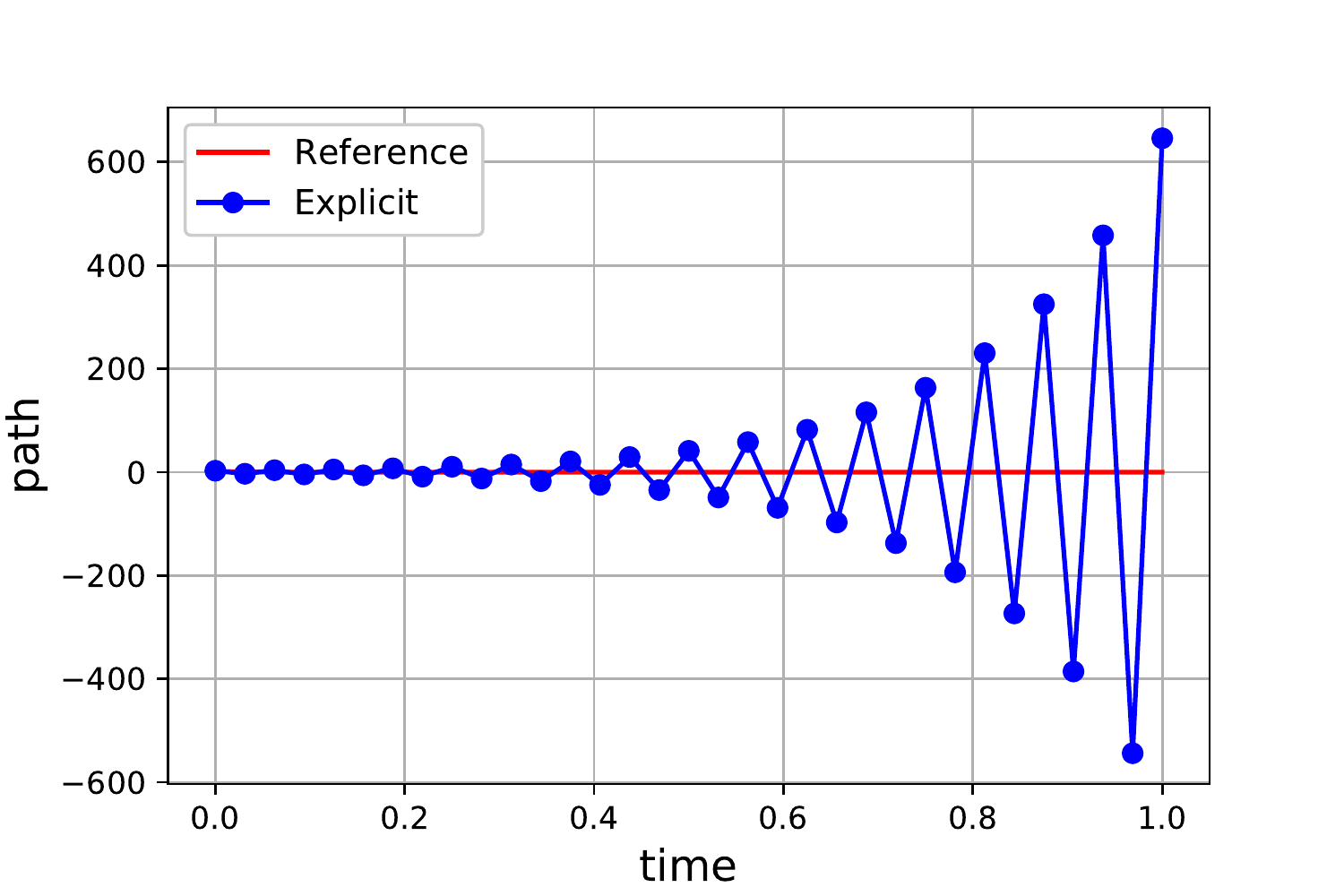}
    }\\
    \caption{Comparison between implicit and explicit schemes of RDE
    \eqref{eq:RDEexample3}.} 
    \label{fig8}
  \end{figure}

  We illustrate this behavior in Figure~\ref{fig8}, where we plot 
  a trajectory with both schemes, the explicit and the implicit Euler methods.
  Both subfigures (a) and (b) show a reference solution of
  \eqref{eq:RDEexample3} with step size $h_{\text{ref}}=2^{-14}$ 
  and a further trajectory for each method with the coarser step size $h =
  2^{-5}$. Observe that this step size violates the condition
  \eqref{eqn:1dmap} for the explicit Euler method.

  In part (a) we see that the implicit Euler method already gives
  a rather good approximation of the reference solution for this step size.
  On the other hand, we observe in part (b) that the trajectory of the explicit
  method exhibits strong oscillations which are purely artificially induced
  and, therefore, undesirable. 

  In particular, this observation is of particular importance
  if the numerical method is embedded, for instance, in a multilevel Monte
  Carlo algorithm. Here the effectiveness of the multilevel algorithm depends
  on the availability of a one-step method which also behaves stable for rather
  coarse step sizes. For an analysis of the multilevel Monte Carlo algorithm in
  the context of rough differential equations we refer to \cite{BFRS16}.

\end{example}

\begin{example}
We consider the following 2-dim RDE driven by 2d fractional Brownian motion
\begin{align}
  \label{eq:RDEexample4}
  \begin{split}
    \begin{cases}
      \diff{y} = (y-|y|^2y)\diff{y} +
      \sigma_1(y)\diff{B^1_{H_1}}+\sigma_2(y)\diff{B^2_{H_2}}, \quad t \in (0,1],\\
      y(0) =[10.0,-10.0]^T,
    \end{cases}
  \end{split}
\end{align}
where $[B^{H_1}, B^{H_2}]^T$ is a 2-dim fractional Brownian motion with Hurst parameter $H_1, H_2\in (0,1)$ in each direction, and $\sigma_1(y)=[\cos(y_2),-0.9-10\cos(y_1)]^T$ and $\sigma_2(y)=[\cos(|y|),0]^T$. From \cite[Theorem 4.3]{RS17}, it follows that the equation defines a semiflow and that the assumptions of Theorem \ref{thm:rates_Gaussian_rough} are satisfied.

We are free to choose the Hurst parameter $H\in (\frac{1}{3},\frac{1}{2}]$. We pick $H_1=H_2=\frac{5}{12}$ in the experiment. We then simulate the solutions via 
the implicit Milstein scheme with stepsizes $h\in \{2^{-7},2^{-8}, 2^{-9},2^{-10},2^{-11},2^{-12}\}$ and compare them to the reference solution obtained via a finer step size of $h_{\mathrm{ref}} = 2^{-14}$.

In Figure~\ref{fig6} (a), we consider the RDE \eqref{eq:RDEexample4} and plot the 
pathwise errors against the underlying step size, i.e., the number $n$
on the $x$-axis indicates the corresponding simulation is based on the step
size $h = 2^{-n}$. The finest step size here is $2^{-12}$. In addition, average EOC obtained is $0.38$, is larger compared to the expected order of convergence $\frac{3}{q}-1<\frac{3\times 5}{12}-1=0.25$ from Theorem \ref{thm:rates_Gaussian_rough}.
In Figure \ref{fig6} (b), two path simulations is also demonstrated for the performance of explicit and implicit Milstein scheme with stepsize $h=2^{-7}$. Clearly numerical solutions from the implicit scheme gives a better approximation. 

In addition, we also test the corresponding forward and backward Euler scheme of RDE \eqref{eq:RDEexample4} with even coarser stepsize $h_o=2^{-6}$. Forward Euler scheme returns an overflow error, which indicates an explosion of the solution of forward Euler scheme. Though eventually both forward and backward Euler schemes may give the same order of convergence, backward Euler outperforms backward one with coarser stepsizes. This observation is particularly crucial for easing computational burden of large-scaled simulations and thus has practical impact on computational cost.

\begin{figure}
\centering
\subfigure[a][Pathwise errors]{\includegraphics[width=0.45\textwidth]{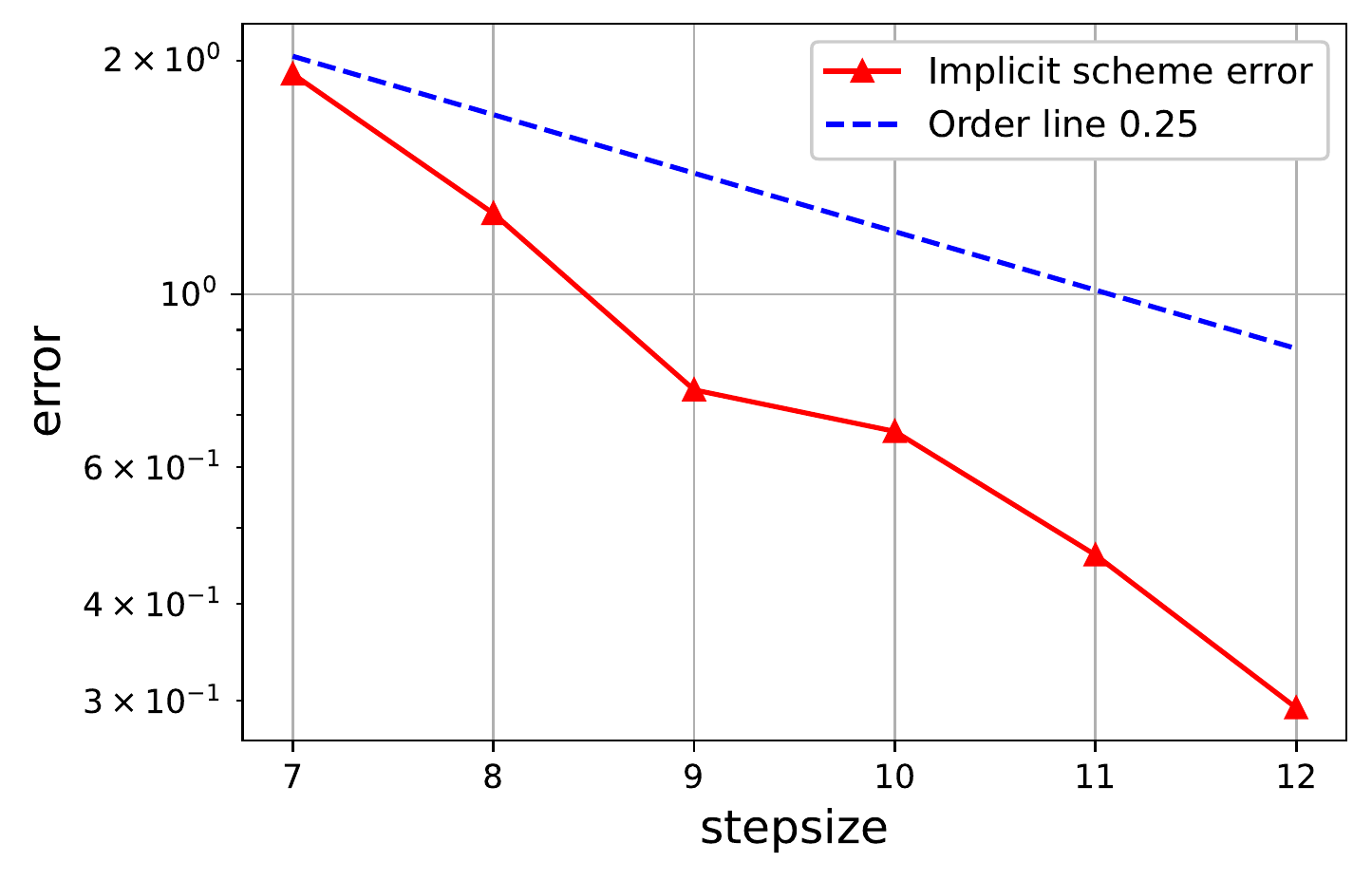}}\qquad
\subfigure[b][Path simulation with $h=2^{-7}$]{\includegraphics[width=0.5\textwidth]{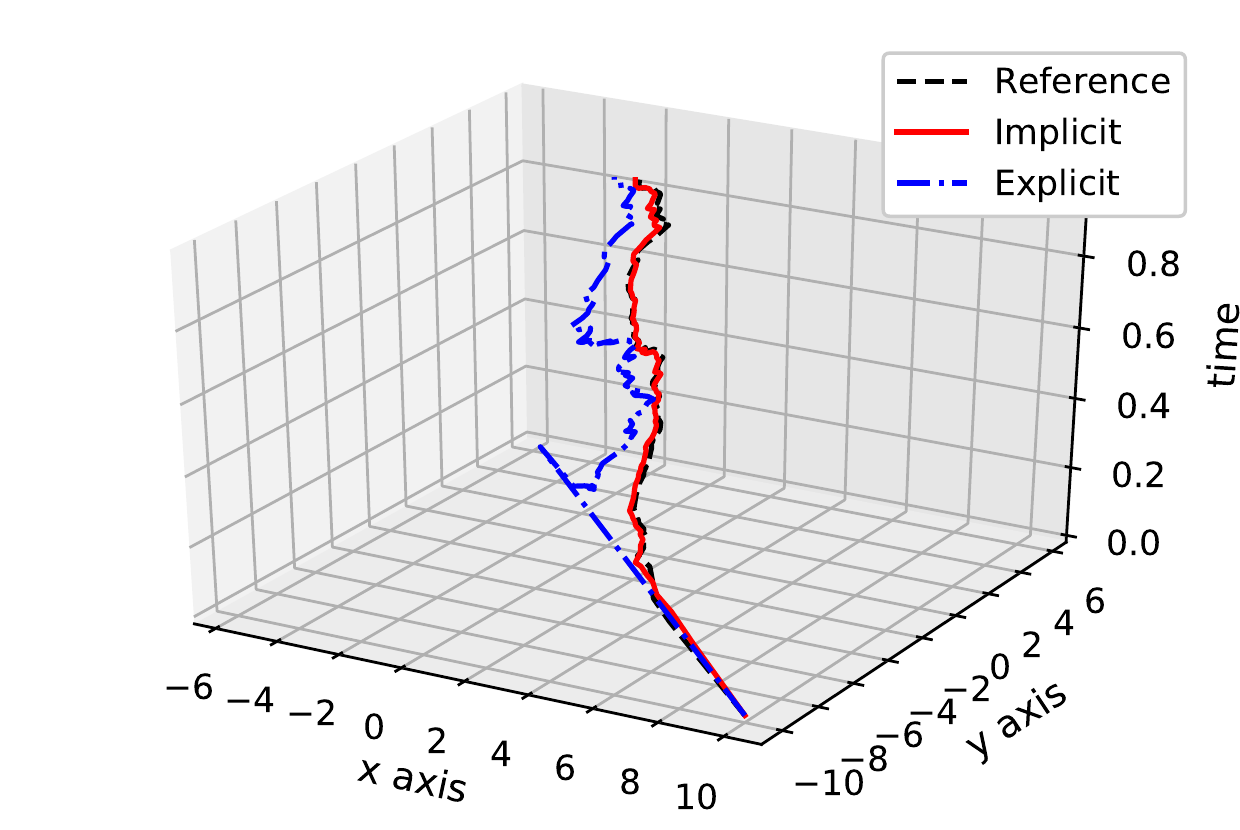}}\\
\caption{Numerical experiment for 2-dim RDE \eqref{eq:RDEexample4}.}
\label{fig6}
\end{figure}
\end{example}

\subsection*{Acknowledgements}
\label{sec:acknowledgements}

SR is supported by the MATH+ project AA4-2 \textit{Optimal control in energy markets using rough analysis and deep networks}. YW would acknowledge Alan Turing Institute for funding this work through EPSRCgrant EP/N510129/1 and EPSRC through the project EP/S2026347/1, titled \textit{Unparameterised multi-modal data, high order signature, and the mathematics of data science}. Work on this paper was started while SR and YW were supported by the DFG via Research Unit FOR 2402.

\bibliographystyle{alpha}
\bibliography{numerics_stiff_RDEs}

\end{document}